\documentclass[a4paper]{amsart}
\usepackage[english]{babel}
\usepackage[a4paper,top=3cm,bottom=2cm,left=3cm,right=3cm,marginparwidth=1.75cm]{geometry}
\usepackage{amsmath}
\numberwithin{equation}{section}

\usepackage{graphicx}
\usepackage[colorinlistoftodos]{todonotes}
\usepackage[colorlinks=true, allcolors=blue]{hyperref}
\usepackage{graphicx}
\usepackage{lmodern}
\usepackage[english]{babel}
\usepackage{amsmath}
\usepackage{color}
\usepackage{amsfonts}
\usepackage{amssymb}
\usepackage{amsthm}
\usepackage{stmaryrd}

\usepackage{pstricks}
\usepackage{epsfig}
\newtheorem{teo}{Theorem}[section]

\newtheorem{pro}[teo]{Proposition}
\newtheorem{lem}[teo]{Lemma}
\newtheorem{cor}[teo]{Corollary}
\newtheorem{rem}[teo]{Remark}
\title[Support of the Brown measure of $Y_tP$]{Support of the Brown measure of the product of a free unitary Brownian motion by a free self-adjoint projection}
\author[N. Demni]{Nizar Demni}
\address{IRMAR, Universit\'e de Rennes 1\\ Campus de
Beaulieu\\ 35042 Rennes cedex\\ France}
\email{nizar.demni@univ-rennes1.fr}

\author[T. Hamdi]{Tarek Hamdi}
\address{Department of Management Information Systems \\ College of Business Management \\ Qassim University \\ Ar Rass \\ Saudi Arabia
and Laboratoire d'Analyse Math\'ematiques et applications LR11ES11 \\ Universit\'e de Tunis El-Manar \\ Tunisie}
\email{ t.hamdi@qu.edu.sa } 

\usepackage{graphicx}

\begin{document}

\maketitle

\begin{abstract}
The first part of this paper is devoted to the Brown measure of the product of the free unitary Brownian motion by an arbitrary free non negative operator. Our approach follows the one recently initiated by Driver-Hall-Kemp though there are substantial differences at the analytical side. In particular, the corresponding Hamiltonian system is completely solvable and the characteristic curve describing the support of the Brown measure has a non-constant (in time) argument. In the second part, we specialize our findings to the product of the free unitary Brownian motion by a free self-adjoint projection and obtain an explicit description of its support.
\end{abstract}

\tableofcontents

\section{Introduction}
Let  $(\mathcal{A}, \tau)$ be a $W^{\star}$-probability space, that is, $\mathcal{A}$ is a von Neumann algebra with a faithful tracial state $\tau$ and  unit ${\bf 1}$. To an arbitrary $a \in \mathcal{A}$ is associated its Fuglede-Kadison determinant (\cite{Fug-Kad}): 
\begin{equation*}
\Delta(a) := \exp[\tau(\log(|a|))] = \exp \int_{\mathbb{R}} \log(u) \mu_{|a|}(dt) \quad \in [0, +\infty[,
\end{equation*}
where $|a| = (a^{\star}a)^{1/2}$ is the radial part of $a$ and $\mu_{|a|}$ is the spectral measure of $|a|$. If we set 
\begin{equation*}
L(a) := \ln(\Delta(a)) \in [-\infty, +\infty[, 
\end{equation*}
then the map $\lambda \mapsto L(a-\lambda {\bf 1})$ is subharmonic on $\mathbb{C}$ and harmonic on the resolvent set of $a$ (\cite{Brown}). As a matter of fact, the Riesz decomposition Theorem gives rise to a probability measure $\nu_a$ supported in the spectrum of $a$ and called the Brown measure of $a$. Concretely, it is given by 
\begin{equation*}
\frac{1}{2\pi}\nabla^2 L(a-\lambda {\bf 1}) 
\end{equation*}
in the distributional sense and is uniquely determined among all compactly-supported measure by the identity 
\begin{equation*}
L(a-\lambda {\bf 1})  = \int_{\mathbb{C}} \ln(|\lambda - z|)\nu_a(dz). 
\end{equation*}   
Using a regularization argument for the Fuglede-Kadison determinant, the Brown measure can be computed as (see e.g. \cite{Min-Spe}): 
\begin{equation*}
\frac{1}{4\pi} \nabla^2 \lim_{x \rightarrow 0^+} \tau[\log(|a-\lambda {\bf 1}|^2 + x)] = \frac{1}{4\pi} \lim_{x \rightarrow 0^+} \nabla^2\tau[\log(|a-\lambda {\bf 1}|^2 + x)],
\end{equation*}
where the limit in the RHS is in the weak sense. This formula offers the opportunity to use analytical techniques to compute the Brown measure  of $a$ since 
\begin{equation*}
x \mapsto S_a(\lambda, x):= \tau[\log(|a-\lambda {\bf 1}|^2 + x)] 
\end{equation*}
defines an analytic function in the right half-plane and may be expanded for large $|x|$ into a generating function for the moments of $|a-\lambda {\bf 1}|^2$. Actually, if $a$ is a free It\^o process (i.e. solution of a free stochastic differential equation) then such expansion may be turned into a partial differential equation (PDE). This is for instance valid for the free multiplicative Brownian motion and for its additive (circular) counterpart, and led in \cite{DHK} and in \cite{Ho-Zha} respectively to the full description of the corresponding Brown measures.

In this paper we adapt the approach initiated in \cite{DHK} to the partial isomety $Y_tP, t \geq 0$. Here, $Y = (Y_t)_{t\geq0}$ is a free unitary free Brownian motion (\cite{Biane}) and $P$ is a self-adjoint projection in $(\mathcal{A}, \tau)$ with rank $\tau(P) \in (0,1)$ and  free from $Y$. Nonetheless, a large part of our computations applies to operators of the form $Y_th$ where $h$ is a non negative operator free from $Y$, which are the natural dynamical analogues of R-diagonal operators (\cite{Nic-Spe}). Using free stochastic calculus, we derive a nonlinear first-order PDE for the map 
  \begin{equation*}
(t, \lambda, x) \mapsto S(t,\lambda,x):= S_{Y_th}(\lambda, x) =\tau\big[ \log ((Y_th-\lambda)^{\star}(Y_th-\lambda)+x)\big],
\end{equation*}
and write down the corresponding Hamiltonian system of coupled ordinary differential equations (hereafter ode). It turns out that the latter is completely solvable: the characteristic curve 
\begin{equation*}
u \mapsto \lambda(u), \quad u \geq 0,
\end{equation*}
is explicitly determined and allows to solve all the remaining ODEs.  However, for ease of reading, we shall only write down those curves needed for the description of the support of the Brown measure. In particular, we determine the blow-up time of the solution of the Hamiltonian system. Compared to the system studied in \cite{DHK}, the angular momentum is still a constant of motion, yet the argument of the characteristic curve $u \mapsto \lambda(u)$ is no longer constant and is rather affine in time. As to its radius, it solves a non-linear second order ODE so that its expression depends also on its initial speed. The aim of these computations is the following expression of $S$ along the characteristic curves $u \mapsto \lambda(u), u \mapsto x(u),$ valid up to the blow-up time.
 \begin{teo}\label{expression}
As long as the solution of the Hamiltonian system exist, we have
\begin{align}\label{S}
S(u,\lambda(u),x(u))&=\tau(\log(|h-\lambda_0|^2+x_0))+\left(H_0-\frac{1}{2}\right)u+\log|\lambda(u)|-\log|\lambda_0|, 
\end{align}
where we simply write $\lambda_0 := \lambda(0), x_0:= x(0),$ and $H_0$ is the initial value of the Hamiltonian.
\end{teo}
Once this formula is obtained, we specialize our computations to the non normal operator $Y_tP$. Our interest in this particular case is partly motivated by the Brown measure of its (strong) limit as $t \rightarrow \infty$ which was completely determined in \cite{Haa-Lar} using the $S$-transform of $P$ (see also \cite{Bia-Leh} for further instances of explicit Brown measures). Another motivation stems from the study of operators of the form $PYQ$, where $Q$ is a selfadjoint projection which is free from $Y$, which are large-size limits (in the sense of mixed moments) of truncations of the Brownian motion in the unitary group. In particular, their limit as $t \rightarrow \infty$ are truncations of Haar unitary matrices whose densities, when they exist, were computed in \cite{Ols}. Even more, the eigenvalues densities of square truncations were used in \cite{Som-Zcy} to analyze statistical properties of random quantum channels exhibiting a chaotic scattering. In this respect, it was further proved in \cite{Pet-Ref} that the empirical measure of any square truncation converges weakly to the Brown measure of $PUP$ in the compressed algebra. Up to an atom at zero and a normalizing constant, the latter coincides with the Brown measure of $UP$ since $\ker(UP) = \ker(P)$ and this coincidence remains valid at finite time $t$ (see \cite{Haa-Lar}, p.350). As to the singular values of (arbitrary) truncations of Haar-unitary matrices and of the unitary Brownian motion, they are squares of the eigenvalues of matrices from the Jacobi unitary ensemble and of the Hermitian Jacobi process respectively. These two matrix models converge in the large-size limit to the free Beta distribution and to the free Jacobi process respectively, and we refer the reader to the papers \cite{Col} and \cite{Ham} for further details and references.

Coming back to the operator $PY_tP$, the main result of this paper provides the following description of the support of its Brown measure $\mu_{(PY_tP)}$:
\begin{teo}\label{support}
Set $\tau(P) :=\alpha\in(0,1)$. Then for any $t > 0$, the support of $\mu_{(PY_tP)}$ is contained in the region enclosed by the Jordan curve $ f_{t,\alpha}\big(F_{t,\alpha} \big)$, where
\begin{align*}
f_{t,\alpha}(z) :=ze^{t\frac{2\alpha-1+z}{2(1-z)}},
\end{align*}
and $F_{t,\alpha}$ is the closure of the set:
\begin{equation*}
\left\{\lambda \in \mathbb{C}: |1-\alpha-\lambda|\ne \alpha, \quad |f_{t,\alpha}(\lambda)|^2 =\frac{\alpha|\lambda|^2}{\alpha|\lambda|^2+(1-\alpha)|1-\lambda|^2} \right\},
\end{equation*}
and is a Jordan curve. 
\end{teo}

\begin{figure}
\begin{center}
\includegraphics[width=8cm]{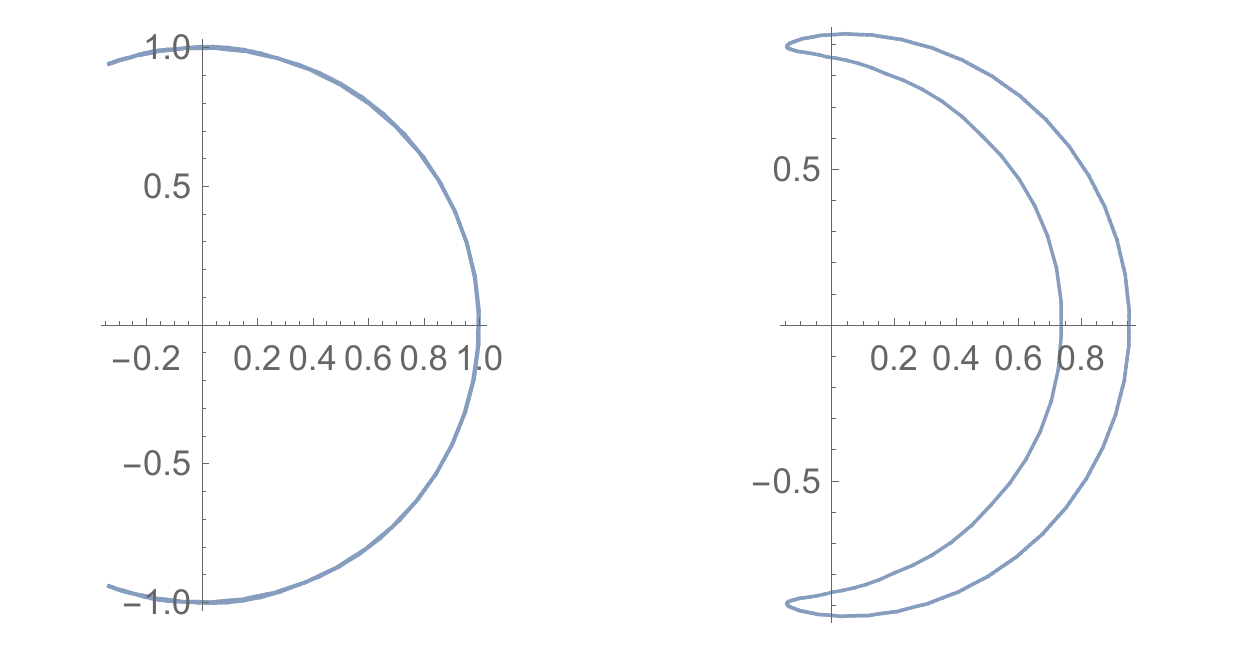} 
\end{center}
\caption{The curve $f_{t,\alpha}(F_{t,\alpha})$ with $(t,\alpha)=(1,1)$ (left) and $(t,\alpha)=(1,0.8)$ (right)}.
\end{figure}

The proof of this theorem is a straightforward consequence of an explicit expression of the function
\begin{equation*}
s_t(\lambda)=\lim_{x\rightarrow0}S(t,\lambda,x), 
\end{equation*}
when $\lambda_0$ lies in some region outside the closure $\overline{\Sigma}_{t,\alpha}$ of the bounded component $\Sigma_{t,\alpha}$ of the complementary of $F_{t,\alpha}$. Actually, these initial values of the characteristic curve $u \mapsto \lambda(u)$ allows to let $x_0$ approaches zero in the expression \eqref{S}, in which case the whole curve $u \mapsto x(u)$ vanishes and the solution of the Hamiltonian system exists up to time $t$.  
On the other hand, we shall retrieve Haagerup-Larsen result for the operator $PUP$: the boundary of the support of $\mu_{(PY_tP)}$ approaches the circle $\mathbb{T}(0,\sqrt{\alpha})$ as $t$ approaches infinity. Let us finally point out that the affine time-dependence of the argument of the curve $u \mapsto \lambda(u)$ makes the description of the non-atomic part of $\mu_{(Y_tP)}$ far from being accessible. For that reason, we postpone this task to a future research work. We would like also to stress that the rank $\alpha = \tau(P)$ varies in $(0,1)$ since some results proved in the sequel do not extend to the value $\alpha = 1$ corresponding to $Y_t$ whose spectrum is already known (\cite{Biane1}).

The paper is organized as follows. In the next section, we perform the preliminary computations and the analysis of the Hamiltionian system associated with the operator $Y_th$, leading to the proof of Theorem \ref{expression}. Section 3 is devoted to the particular case $h = P$. There, we firstly supply a parametrization of $F_{t,\alpha}$ and prove that its image under $f_{t,\alpha}$ is a Jordan curve. Afterwards, we prove that any $\lambda$ lying outside the Jordan domain delimited by $f_{t,\alpha}(F_{t,\alpha})$ is attainable by some characteristic curve starting at $\lambda_0 \notin \overline{\Sigma}_{t,\alpha}$ and derive the explicit expression of $s_t(\lambda)$. The latter is, up to a linear combination of logarithmic potentials of two dirac measures, the real part of a holomorphic function.

\section{Hamiltonian system for $Y_th$}
\subsection{The PDE for $S$}
For the sake of clarity, we introduce the following notations:
\begin{equation*}
a_t = a_t(h)  := Y_th,
\end{equation*}
\begin{equation*}
a_{t,\lambda} = a_{t,\lambda}(h) := |a_t-\lambda|^2=(a_t-\lambda)^{\star}(a_t-\lambda)=h^2-\overline{\lambda}a_t-\lambda a_t^{\star}+|\lambda|^2,
\end{equation*}
where we omit the dependence on $h$. Using free stochastic calculus, we prove:
\begin{pro} \label{moment}
For any $n \geq 1$,
\begin{align*}
\frac{d}{dt}\tau(a_{t,\lambda}^{n})=&\frac{n}{2}\tau\big(a_{t,\lambda}^{n-1}(\overline{\lambda}a_t+\lambda a_t^*)\big)+n|\lambda|^2\sum_{j=0}^{n-2}\tau(h^2a_{t,\lambda}^{j})\tau\big( a_{t,\lambda}^{n-2-j}\big)
\\&-\frac{n}{2}\sum_{j=0}^{n-2}\overline{\lambda}^2\tau(a_ta_{t,\lambda}^{j})\tau\big(a_ta_{t,\lambda}^{n-2-j}\big)+\lambda^2\tau(a_t^{\star}a_{t,\lambda}^{j})\tau\big(a_t^{\star}a_{t,\lambda}^{n-2-j} \big),
\end{align*}
where an empty sum is zero.
\end{pro}
\begin{proof}
Since
\begin{equation*}
da_{t,\lambda}=-\overline{\lambda}da_t-\lambda da_t^{\star}=-\overline{\lambda}dY_th-\lambda hdY_t^{\star},
\end{equation*}
then the free It\^o's formula entails:
\begin{equation}\label{mom1}
d(a_{t,\lambda}^n) =\sum_{k=1}^na_{t,\lambda}^{k-1}da_{t,\lambda}a_{t,\lambda}^{n-k}+\sum_{1\leq j<k\leq n} a_{t,\lambda}^{j-1}\left(da_{t,\lambda}a_{t,\lambda}^{k-j-1}da_{t,\lambda}\right)a_{t,\lambda}^{n-k},
\end{equation}
where the term inside the brackets is a free semi-martingale bracket. Moreover, it is known that (see e.g. \cite{Biane}):
\begin{equation*}
dY_{t} = iY_tdX_t-\frac{1}{2}Y_tdt, \quad \quad dY_{t}^{\star}=-idX_tY_t^{\star}-\frac{1}{2}Y_t^{\star}dt,
\end{equation*}
where $(X_t)_{t \geq 0}$ is a free additive Brownian motion. Consequently, 
\begin{align*}
da_{t,\lambda}&=-\overline{\lambda}(iY_tdX_t-\frac{1}{2}Y_tdt)h-\lambda h(-idX_tY_t^{\star}-\frac{1}{2}Y_t^{\star}dt)
\\&= i(-\overline{\lambda}Y_tdX_th+\lambda h dX_tY_t^{\star})+\frac{1}{2}\big(\overline{\lambda}a_t+\lambda a_t^{\star} \big)dt.
\end{align*}
Set $dw_t :=-\overline{\lambda}Y_tdX_th+\lambda h dX_tY_t^*$. Then, for any adapted process $(\kappa_t)_{t \geq 0}$, 
\begin{align*}
dw_t\kappa_t dw_t&=\overline{\lambda}^2\tau(h\kappa_tY_t)a_tdt+\lambda^2\tau(Y_t^*\kappa_th)a_t^{\star}dt-|\lambda|^2\tau(Y_t^{\star}\kappa_tY_t)h^2dt-|\lambda|^2\tau(h\kappa_th)dt
\\&=\overline{\lambda}^2\tau(a_t\kappa_t)a_tdt+\lambda^2\tau(a_t^{\star}\kappa_t)a_t^{\star}dt-|\lambda|^2\tau(\kappa_t)h^2dt-|\lambda|^2\tau(h^2\kappa_t)dt
\end{align*}
Specializing the last equation to $\kappa_t = a_{t,\lambda}^{k-j-1}$ and taking the trace in both sides of \eqref{mom1}, we get: 
\begin{multline*}
\frac{d}{dt}\tau(a_{t,\lambda}^{n})= \frac{1}{2}\sum_{k=0}^{n-1}\tau\big(a_{t,\lambda}^{k}(\overline{\lambda}a_t+\lambda a_t^{\star})a_{t,\lambda}^{n-1-k}\big)
+|\lambda|^2\sum_{i=0}^{n-2}\sum_{j=0}^{n-2-i}\tau\left(a_{t,\lambda}^{i}h^2a_{t,\lambda}^{n-2-i-j}\tau(a_{t,\lambda}^{j})\right.  \\\left. + \tau(h^2a_{t,\lambda}^{j}) a_{t,\lambda}^{n-2-j}\right)
 -\sum_{i=0}^{n-2}\sum_{j=0}^{n-2-i}\tau\left(\overline{\lambda}^2\tau(a_ta_{t,\lambda}^{j})a_{t,\lambda}^{i}a_ta_{t,\lambda}^{n-2-i-j}+\lambda^2\tau(a_t^{\star}a_{t,\lambda}^{j})a_{t,\lambda}^{i}a_t^{\star}a_{t,\lambda}^{n-2-i-j} \right)
\\= \frac{n}{2}\tau\big(a_{t,\lambda}^{n-1}(\overline{\lambda}a_t+\lambda a_t^*)\big)+|\lambda|^2\sum_{i=0}^{n-2}\sum_{j=0}^{n-2-i}\tau\big(h^2a_{t,\lambda}^{n-2-j}\big)\tau(a_{t,\lambda}^{j}) +\tau(h^2a_{t,\lambda}^{j})\tau\big( a_{t,\lambda}^{n-2-j}\big)
\\-\sum_{i=0}^{n-2}\sum_{j=0}^{n-2-i}\overline{\lambda}^2\tau(a_ta_{t,\lambda}^{j})\tau\big(a_ta_{t,\lambda}^{n-2-j}\big)+\lambda^2\tau(a_t^{\star}a_{t,\lambda}^{j})\tau\big(a_t^{\star}a_{t,\lambda}^{n-2-j} \big)
\\= \frac{n}{2}\tau\big(a_{t,\lambda}^{n-1}(\overline{\lambda}a_t+\lambda a_t^{\star})\big)+n|\lambda|^2\sum_{j=0}^{n-2}\tau(h^2a_{t,\lambda}^{j})\tau\big( a_{t,\lambda}^{n-2-j}\big)
\\ -\frac{n}{2}\sum_{j=0}^{n-2}\overline{\lambda}^2\tau(a_ta_{t,\lambda}^{j})\tau\big(a_ta_{t,\lambda}^{n-2-j}\big)+\lambda^2\tau(a_t^{\star}a_{t,\lambda}^{j})\tau\big(a_t^{\star}a_{t,\lambda}^{n-2-j} \big).
\end{multline*}
\end{proof}

Write   
\begin{equation*}
S(t,\lambda,x):= S_{Y_th}(\lambda, x) =\tau\big[ \log (a_{t,\lambda}+x)\big].
\end{equation*}
Using the previous proposition, we get
\begin{cor}\label{partial}
If $\Re(x) > 0$, then the function $S$ satisfies
\begin{align*}
 \frac{\partial S}{\partial t}=&\frac{1}{2}\tau\big((\overline{\lambda}a_t+\lambda a_t^{\star})(a_{t,\lambda}+x)^{-1}\big)-|\lambda|^2\tau\big( h^2(a_{t,\lambda}+x)^{-1}\big)\tau\big((a_{t,\lambda}+x)^{-1}\big)
\\&+\frac{\overline{\lambda}^2}{2}[\tau\big( a_t(a_{t,\lambda}+x)^{-1}\big)]^2
+\frac{\lambda^2}{2}[\tau\big( a_t^{\star}(a_{t,\lambda}+x)^{-1}\big)]^2.
\end{align*}
\end{cor}
\begin{proof}
For large $|x|$, we expand $S$ into power series
\begin{equation*}
    S(t,\lambda,x)=\log x+\sum_{n=1}^\infty \frac{(-1)^{n-1}}{nx^n} \tau\big[( a_{t,\lambda})^n\big],
\end{equation*}
and differentiate it termwise with respect to $t$. Using Proposition \ref{moment}, we get
\begin{align*}
 \frac{\partial S}{\partial t}=&\frac{1}{x}\frac{d}{dt}\tau(a_{t,\lambda})+\frac{1}{2}\sum_{n=2}^{\infty}\frac{(-1)^{n-1}}{x^n}\tau\big(a_{t,\lambda}^{n-1}(\overline{\lambda}a_t+\lambda a_t^{\star})\big)
 +|\lambda|^2\sum_{n=2}^{\infty}\frac{(-1)^{n-1}}{x^n}\sum_{j=0}^{n-2}\tau(h^2a_{t,\lambda}^{j})\tau\big( a_{t,\lambda}^{n-2-j}\big)
\\&-\frac{1}{2}\sum_{n=2}^{\infty}\frac{(-1)^{n-1}}{x^n}\sum_{j=0}^{n-2}\overline{\lambda}^2\tau(a_ta_{t,\lambda}^{j})\tau\big(a_ta_{t,\lambda}^{n-2-j}\big)+\lambda^2\tau(a_t^{\star}a_{t,\lambda}^{j})\tau\big(a_t^{\star}a_{t,\lambda}^{n-2-j} \big)
\\=&\frac{1}{2}\tau\big((a_{t,\lambda}+x)^{-1}(\overline{\lambda}a_t+\lambda a_t^{\star})\big)
-|\lambda|^2\left(\frac{1}{x}\sum_{k=0}^{\infty}\frac{(-1)^{k}}{x^k} \tau\big( h^2a_{t,\lambda}^{k}\big)\right)\left(\frac{1}{x}\sum_{l=0}^{\infty}\frac{(-1)^{l}}{x^l} \tau\big( a_{t,\lambda}^{l}\big)\right)
\\&+\frac{\overline{\lambda}^2}{2}\left(\frac{1}{x}\sum_{l=0}^{\infty}\frac{(-1)^{l}}{x^l} \tau\big(a_t a_{t,\lambda}^{l}\big)\right)^2+\frac{\lambda^2}{2}\left(\frac{1}{x}\sum_{l=0}^{\infty}\frac{(-1)^{l}}{x^l} \tau\big(a_t^{\star} a_{t,\lambda}^{l}\big)\right)^2
\\=&\frac{1}{2}\tau\big((a_{t,\lambda}+x)^{-1}(\overline{\lambda}a_t+\lambda a_t^{\star})\big)-|\lambda|^2\tau\big( h^2(a_{t,\lambda}+x)^{-1}\big)\tau\big((a_{t,\lambda}+x)^{-1}\big)
\\&+\frac{\overline{\lambda}^2}{2}[\tau\big( a_t(a_{t,\lambda}+x)^{-1}\big)]^2 +\frac{\lambda^2}{2}[\tau\big( a_t^{\star}(a_{t,\lambda}+x)^{-1}\big)]^2.
\end{align*}
Since $S$ and the expression displayed in the right-hand side of the last equality are analytic in the right half-plane, the corollary follows. 
\end{proof}
The analyticity of $x \mapsto S(t,\lambda,x)$ is only needed to prove Corollary \ref{partial}. Henceforth, we shall assume $x > 0$ and derive the following PDE for $S$. 
\begin{teo}
The function $S$ satisfies:
\begin{align*}
    \frac{\partial S}{\partial t}=&x|\lambda|^2\left(\frac{\partial S}{\partial x}\right)^2+\frac{\overline{\lambda}^2}{2}\left(\frac{\partial S}{\partial \overline{\lambda}}\right)^2
+\frac{\lambda^2}{2}\left(\frac{\partial S}{\partial \lambda}\right)^2
-\frac{\overline{\lambda}}{2}\frac{\partial S}{\partial \overline{\lambda}}
-\frac{\lambda}{2}\frac{\partial S}{\partial \lambda},
\end{align*}
with the initial value:
\begin{equation*}
    S(0,\lambda,x)=\tau(\log(|h-\lambda|^2+x)).
\end{equation*}
Equivalently, if $\lambda=a+ib$ then this PDE reads:
\begin{align}\label{pde2}
    \frac{\partial S}{\partial t}=&x|\lambda|^2\left(\frac{\partial S}{\partial x}\right)^2-\frac{a}{2}\frac{\partial S}{\partial a}
-\frac{b}{2}\frac{\partial S}{\partial b}+ab\frac{\partial S}{\partial a}
\frac{\partial S}{\partial b}+\frac{a^2-b^2}{4}\left(\big(\frac{\partial S}{\partial a}\big)^2-
\big(\frac{\partial S}{\partial b}\big)^2\right).
\end{align}
\end{teo}
\begin{proof}
We appeal to the following identities:
\begin{equation}\label{Id1}
    \frac{\partial S}{\partial x}=\tau\big((a_{t,\lambda}+x)^{-1}\big),
\end{equation}
\begin{equation}\label{Id2}
    \frac{\partial S}{\partial \lambda}=-\tau\big((a_t-\lambda)^{\star}(a_{t,\lambda}+x)^{-1}\big),
\end{equation}
\begin{equation}\label{Id3}
    \frac{\partial S}{\partial \overline{\lambda}}=-\tau\big((a_t-\lambda)(a_{t,\lambda}+x)^{-1}\big),
\end{equation}
which are straightforward consequences of \cite[Lemma 1.1]{Brown}. It follows that:
\begin{equation*}
   \tau\big(a_t(a_{t,\lambda}+x)^{-1}\big) =\lambda \frac{\partial S}{\partial x}-\frac{\partial S}{\partial \overline{\lambda}},\qquad \tau\big(a_t^{\star}(a_{t,\lambda}+x)^{-1}\big) =\overline{\lambda} \frac{\partial S}{\partial x}-\frac{\partial S}{\partial \lambda}.
\end{equation*}
Besides, writing $h^2=a_t^{\star}a_t=(a_t-\lambda+\lambda)^{\star}(a_t-\lambda+\lambda)$, we get
\begin{align*}
 \tau\big( h^2(a_{t,\lambda}+x)^{-1}\big)=& \tau\big( a_{t,\lambda}(a_{t,\lambda}+x)^{-1}\big)+\lambda\tau\big( (a_t-\lambda)^{\star}(a_{t,\lambda}+x)^{-1}\big)
 \\&+\overline{\lambda}\tau\big( (a_t-\lambda)(a_{t,\lambda}+x)^{-1}\big)+|\lambda|^2\tau\big( (a_{t,\lambda}+x)^{-1}\big)  
 \\=&1+(|\lambda|^2-x)\frac{\partial S}{\partial x}-\lambda\frac{\partial S}{\partial \lambda}-\overline{\lambda}\frac{\partial S}{\partial \overline{\lambda}}.
\end{align*}
Together with Corollary \ref{partial}, we end up with:
\begin{align*}
    \frac{\partial S}{\partial t}=&\frac{\overline{\lambda}}{2}[\lambda \frac{\partial S}{\partial x}-\frac{\partial S}{\partial \overline{\lambda}}]+\frac{\lambda}{2}[\overline{\lambda} \frac{\partial S}{\partial x}-\frac{\partial S}{\partial \lambda}]-|\lambda|^2\frac{\partial S}{\partial x}[1+(|\lambda|^2-x)\frac{\partial S}{\partial x}-\lambda\frac{\partial S}{\partial \lambda}-\overline{\lambda}\frac{\partial S}{\partial \overline{\lambda}}]
   \\&+\frac{\overline{\lambda}^2}{2}[\lambda \frac{\partial S}{\partial x}-\frac{\partial S}{\partial \overline{\lambda}}]^2
+\frac{\lambda^2}{2}[\overline{\lambda} \frac{\partial S}{\partial x}-\frac{\partial S}{\partial \lambda}]^2
\\=&x|\lambda|^2\left(\frac{\partial S}{\partial x}\right)^2+\frac{\overline{\lambda}^2}{2}\left(\frac{\partial S}{\partial \overline{\lambda}}\right)^2
+\frac{\lambda^2}{2}\left(\frac{\partial S}{\partial \lambda}\right)^2
-\frac{\overline{\lambda}}{2}\frac{\partial S}{\partial \overline{\lambda}}
-\frac{\lambda}{2}\frac{\partial S}{\partial \lambda},
\end{align*}
as desired.
\end{proof}

\subsection{Hamilton equations}
Since \eqref{pde2} is a nonlinear and first-order PDE, then it is natural to appeal to the method of characteristics. Even more, it turns out that the Hamiltonian formalism is well-suited to our situation as it was in \cite{DHK}, which amounts to make a coupling between space variables $a, b, x,$ and the partial derivatives of $S$ with respect to them (momenta $p_a, p_b, p_x$) such that the right-hand side of \eqref{pde2} (the Hamiltonian) viewed as a function of these six variables is constant along the characteristics.     
More precisely, the Hamiltonian corresponding to the PDE \eqref{pde2} is (up to a sign) given by:
\begin{equation*}
H(a,b,x,p_a,p_b,p_x)=-x(a^2+b^2)p_x^2+\frac{a}{2}p_a+\frac{b}{2}p_b-abp_ap_b-\frac{a^2-b^2}{4}(p_a^2-p_b^2).
\end{equation*}
If we require that $a,b,x,p_a,p_b,p_x,$ evolve along curves, then the Hamilton's equations are given by:
\begin{equation*}
    \frac{da}{du}=\frac{\partial H}{\partial p_a}; \frac{db}{du}=\frac{\partial H}{\partial p_b};
    \frac{dx}{du}=\frac{\partial H}{\partial p_x};
\end{equation*}
\begin{equation}\label{sys}
    \frac{dp_a}{du}=-\frac{\partial H}{\partial a}; \frac{dp_b}{du}=-\frac{\partial H}{\partial b};
    \frac{dp_x}{du}=-\frac{\partial H}{\partial x}.
\end{equation}
and ensure that 
\begin{equation*}
\partial_u[H(a(u), b(u), x(u), p_a(u), p_b(u), p_x(u))] = 0.
\end{equation*}
Besides, the initial conditions of the momenta curves are deduced from \eqref{Id1}, \eqref{Id2} and \eqref{Id3}:
\begin{eqnarray*}
p_a(0) & = \displaystyle \frac{\partial S}{\partial a} (a_0, b_0, x_0) & = -2\tau(q_0(h-a_0)), \\
p_b(0) & =  \displaystyle \frac{\partial S}{\partial b} (a_0, b_0, x_0) & = 2b_0\tau(q_0), \\
p_x(0) &=  \displaystyle \frac{\partial S}{\partial x} (a_0, b_0, x_0) & = \tau(q_0),
\end{eqnarray*}
where we simply write 
\begin{equation*}
a(0)=a_0; \quad b(0)=b_0; \quad x(0)=x_0;
\end{equation*}
 and we set
\begin{equation*}
q_0 =((h-a_0)^2+b_0^2+x_0)^{-1}.
\end{equation*}
Consequently, the value of the Hamiltonian along the characteristic curves is 
\begin{align*}
H_0 :=&-x_0(a_0^2+b_0^2)p_x(0)^2-a_0\tau(q_0(h-a_0))+b_0^2\tau(q_0)+4a_0b_0^2\tau(q_0(h-a_0))\tau(q_0)
\\&-(a_0^2-b_0^2)\left( [\tau(p_0h-a_0q_0)]^2-b_0^2\tau(q_0)^2 \right).
\end{align*}
One also checks by direct computations using \eqref{sys} that the following are also constant of motions: 
\begin{pro}\label{ConsMot}
Along any solution of \eqref{sys}, the following quantities remain constant in time:
\begin{itemize}
\item $K_1:= xp_x+\frac{1}{2}(ap_a+bp_b)$,
\item $K_2:= ap_b-bp_a$ (angular momentum),
\item $xp_x^2$.
\end{itemize}
\end{pro}

\subsection{Solving the equations}
Let us write the odes in \eqref{sys} more explicitly: 
\begin{align}
    \dot{a}&=-\frac{a^2-b^2}{2}p_a-a(bp_b-\frac{1}{2});\label{ode1}
    \\\dot{b}&=\frac{a^2-b^2}{2}p_b-b(ap_a-\frac{1}{2}); \label{ode2}
    \\\dot{x}&=-2x(a^2+b^2)p_x;\label{ode3}
    \\\dot{p_a}&=\frac{a}{2}(p_a^2-p_b^2)+p_a(bp_b-\frac{1}{2})+2ax_0p_x(0)^2; \label{ode4}
    \\\dot{p_b}&=-\frac{b}{2}(p_a^2-p_b^2)+p_b(ap_a-\frac{1}{2})+2bx_0p_x(0)^2; \label{ode5}
    \\\dot{p_x}&=(a^2+b^2)p_x^2 \label{ode6}.
\end{align}

From \eqref{ode6}, we obtain
\begin{equation}\label{px}
p_x(u)=\frac{p_x(0)}{1-p_x(0)\int_0^u|\lambda(s)|^2ds}, \quad u \geq 0,
\end{equation}
up to the blow-up time, or equivalently 
\begin{equation*}
\int_0^u|\lambda(s)|^2ds= \frac{1}{p_x(0)}  -\frac{1}{p_x(u)}.
\end{equation*}
Besides, since $xp_x^2$ is a constant of motion then 
\begin{equation*}
x(u)=x_0\left(1-p_x(0)\int_0^u|\lambda(s)|^2ds\right)^2.
\end{equation*}

Now, we come to the following key result: 
\begin{pro}\label{odelambda}
Provided that $\lambda$ does not vanish, it satisfies the following differential equation:
\begin{equation*}
x_0p_x^2(0)|\lambda|^2 + \frac{\ddot{\lambda}\lambda - (\dot{\lambda})^2}{\lambda^2} = 0.
\end{equation*}
\end{pro}
\begin{proof}
Let $\tilde{\lambda}(\cdot):= \lambda(2\cdot)$, then the sum of \eqref{ode1} and \eqref{ode2} gives: 
\begin{equation}\label{E1}
\dot{\tilde{\lambda}} = -p_{\tilde{\lambda}}\tilde{\lambda}^2 + \tilde{\lambda} \quad \Leftrightarrow \quad p_{\lambda} = \frac{\tilde{\lambda} - \dot{\tilde{\lambda}}}{\tilde{\lambda}^2}.
\end{equation}
where we set 
\begin{equation*}
p_{\tilde{\lambda}}(u) := (p_a-ip_b)(2u),
\end{equation*}
while subtracting \eqref{ode5} from \eqref{ode4} gives: 
\begin{equation}\label{E2}
\dot{p}_{\tilde{\lambda}} = (p_{\tilde{\lambda}})^2 \tilde{\lambda} - p_{\tilde{\lambda}} + 4\overline{\tilde{\lambda}}x_0p_x^2(0).
\end{equation}
Differentiating \eqref{E1} and comparing the resulting equation with \eqref{E2}, we further get: 
\begin{equation*}
-\frac{\dot{\tilde{\lambda}}}{\tilde{\lambda}^2} - \frac{\ddot{\tilde{\lambda}}\tilde{\lambda}^2 - 2 \tilde{\lambda}\dot{\tilde{\lambda}}^2}{\tilde{\lambda}^4} = \tilde{\lambda} \frac{\tilde{\lambda}^2 + \dot{\tilde{\lambda}}^2 - 2\tilde{\lambda}\dot{\tilde{\lambda}}}{\tilde{\lambda}^4} + \frac{\dot{\tilde{\lambda}} - \tilde{\lambda}}{\tilde{\lambda}^2} + 4\overline{\tilde{\lambda}}x_0p_x^2(0).
\end{equation*}
Multiplying both sides of the last equation by $\tilde{\lambda}^4$, we equivalently get: 
\begin{equation*}
-\dot{\tilde{\lambda}} \tilde{\lambda}^2 - \ddot{\tilde{\lambda}}\tilde{\lambda}^2 + 2 \tilde{\lambda}\dot{\tilde{\lambda}}^2 = \tilde{\lambda} (\tilde{\lambda}^2 + \dot{\tilde{\lambda}}^2 - 2\tilde{\lambda}\dot{\tilde{\lambda}}) 
+ (\dot{\tilde{\lambda}} - \tilde{\lambda})\tilde{\lambda}^2+ 4|\tilde{\lambda}|^2 \tilde{\lambda}^3x_0p_x^2(0).
\end{equation*}
which reduces after some simplifications to:
\begin{equation*}
4x_0p_x^2(0)|\tilde{\lambda}(u)|^2 + \frac{\ddot{\tilde{\lambda}}\tilde{\lambda} - (\dot{\tilde{\lambda}})^2}{\tilde{\lambda}^2} = 0.
\end{equation*}
Remembering the definition $\tilde{\lambda}(u) = \lambda(2u)$, we are done. 
\end{proof}
Writing 
\begin{equation*}
 \frac{\ddot{\lambda}\lambda - (\dot{\lambda})^2}{\lambda^2} = \ddot{(\ln(\lambda))}
 \end{equation*}
 where $\ln$ is any determination of the logarithm which coincides with the real logarithm on the positive half-line, and setting $\lambda(u) = r(u)e^{i\theta(u)}$, we readily get: 
 \begin{eqnarray*}
 \ddot{\theta} & = & 0 \\ 
 x_0p_x^2(0) r^2 + \ddot{(\ln(r))} & = & 0.
 \end{eqnarray*}
Setting further $r = e^v$, it follows that: 
 \begin{eqnarray*}
 \theta(u) & = & \dot{\theta}(0)u+ \theta_0, \quad \theta_0 := \theta(0),  \\ 
 \ddot{v} + x_0p_x^2(0)e^{2v} & = & 0.
 \end{eqnarray*}
In particular, $\dot{v}$ is decreasing and as such, it is either non positive on the whole interval where it is defined or there exists a time $u_1$ after which it remains non positive. Moreover, 
\begin{equation*}
(\dot{v})^2(u) =  (\dot{v}(0))^2 + x_0p_x^2(0)(e^{2v(0)} - e^{2v(u)}),
\end{equation*}
and using \eqref{ode1} and \eqref{ode2}, we have
\begin{align*}
\dot{v}(0)=\frac{\dot{r}}{r} (0)=\frac{a_0\dot{a}(0)+b_0\dot{b}(0)}{a_0^2+b_0^2}=\frac{1}{2}[1-a_0p_a(0)-b_0p_b(0)].
\end{align*}
Note also that 
\begin{equation*}
(\dot{v}(0))^2 + x_0p_x^2(0)(e^{2v(0)} - e^{2v(t)})\geq 0 \Leftrightarrow 0\leq |\lambda(t)| \leq \sqrt{ |\lambda_0|^2+\frac{\dot{v}(0)^2}{x_0p_x(0)^2}}
\end{equation*}
where $\lambda_0 = a_0 + ib_0$, and that $u_1$ satisfies
\begin{equation*}
 |\lambda(u_1)| = \sqrt{ |\lambda_0|^2+\frac{\dot{v}(0)^2}{x_0p_x(0)^2}}.
\end{equation*}

We shall distinguish two cases: 
\begin{itemize}
\item $\dot{v}$ is non positive ($\dot{v}(0) \leq 0$):  $v$ is decreasing and $|\lambda(t)|\leq |\lambda_0|$ on the whole time-interval where it is defined. Precisely, we have 
\begin{equation*}
\dot{v}(u) = -\sqrt{(\dot{v}(0))^2 +x_0p_x^2(0)(e^{2v(0)} - e^{2v(u)})}.
\end{equation*}
Solving this equation leads to the following result: 
\begin{pro}
Assume $\dot{v}(0) \leq 0$. Then, for any $u \geq 0$, 
\begin{equation*}
|\lambda(u)| =  \frac{2|\lambda_0|e^{\sqrt{C_1}u}(1+\sqrt{1-C_2|\lambda_0|^2})}{C_2|\lambda_0|^2 +e^{2\sqrt{C_1}u}(1+\sqrt{1-C_2|\lambda_0|^2})^2},
\end{equation*}
where

\begin{equation}\label{Constantes}
C_1:= (\dot{v}(0))^2 +x_0p_x^2(0)e^{2v(0)} > 0, \quad C_2 := \frac{x_0p_x^2(0)}{C_1} \in (0,1). 
\end{equation}
\end{pro}
\begin{proof}
We need to compute the primitive: 
\begin{align*}
F(g) : & =  \int^g \frac{dy}{\sqrt{(\dot{v}(0))^2 +x_0p_x^2(0)(e^{2v(0)} - e^{2y})}}, \quad g < v(0) = \ln(r_0), 
\\& =  \int^{e^g} \frac{dy}{\sqrt{C_1}y\sqrt{1-C_2y^2}}
\\& = \int^{\arcsin[\sqrt{C_2}e^g]} \frac{dy}{\sqrt{C_1}\sin(y)}
\\& =  \frac{1}{\sqrt{C_1}}\ln\left(\tan\left[\frac{1}{2}\arcsin[\sqrt{C_2}e^g\right]\right)
\\& =  \frac{1}{\sqrt{C_1}} \ln \left[\frac{\sqrt{C_2}e^g}{1+\sqrt{1-C_2e^{2g}}}\right],
\end{align*}
where the last equality follows from the trigonometric identity: 
\begin{equation*}
\tan\left(\frac{u}{2}\right) = \frac{\sin(u)}{1+\cos(u)}. 
\end{equation*}
As a result, 
\begin{equation*}
F(v(u)) = F(v(0)) - u \geq 0. 
\end{equation*}
But $F$ is invertible and its inverse reads
\begin{equation*}
F^{-1}(g) = \ln\frac{2e^{-\sqrt{C_1}g}}{\sqrt{C_2}(1+e^{-2\sqrt{C_1}g})}.
\end{equation*}
Hence  
\begin{equation*}
v(u) = F^{-1}[F(v(0))- u] = \ln(r(u)),
\end{equation*}
or equivalently 
\begin{equation*}
|\lambda(u)|  = \frac{2e^{\sqrt{C_1}(u - F(v_0))}}{\sqrt{C_2}(1+e^{2\sqrt{C_1}(u-F(v_0))})}.
\end{equation*}
Finally,  
\begin{equation*}
e^{-\sqrt{C_1}F(v_0)} = \frac{1+\sqrt{1-C_2e^{2v_0}}}{\sqrt{C_2}e^{v_0}}, \qquad e^{v_0} = |\lambda_0|, 
\end{equation*}
whence 
\begin{align*}
|\lambda(u)| &= \frac{2e^{\sqrt{C_1}u+v_0}(1+\sqrt{1-C_2e^{2v_0}})}{C_2e^{2v_0} +e^{2\sqrt{C_1}u}(1+\sqrt{1-C_2e^{2v_0}})^2} 
\\& = \frac{2|\lambda_0|e^{\sqrt{C_1}u}(1+\sqrt{1-C_2|\lambda_0|^2})}{C_2|\lambda_0|^2 +e^{2\sqrt{C_1}u}(1+\sqrt{1-C_2|\lambda_0|^2})^2}. 
\end{align*}
The proposition is proved. 
\end{proof}
\item $\dot{v}$ is positive ($\dot{v}(0) > 0$): $v$ is increasing on $(0,u_1)$ and 
\begin{equation*}
|\lambda_0|\leq |\lambda(u)| \leq \sqrt{ |\lambda_0|^2+\frac{\dot{r}(0)^2}{|\lambda_0|^2x_0p_x(0)^2}}
\end{equation*}
 Moreover, we have on that interval: 
\begin{equation*}
\dot{v}(u) = \sqrt{(\dot{v}(0))^2 +x_0p_x^2(0)(e^{2v(0)} - e^{2v(u)})}.
\end{equation*}
Similar computations as above yield:
\begin{pro}
Assume $\dot{v}(0) > 0$. Then, for any $u \in (0,u_1)$, 
\begin{equation*}
|\lambda(u)| =
 \frac{2|\lambda_0|e^{-\sqrt{C_1}u}(1+\sqrt{1-C_2|\lambda_0|^2})}{C_2|\lambda_0|^2 +e^{-2\sqrt{C_1}u}(1+\sqrt{1-C_2|\lambda_0|^2})^2}.
\end{equation*}
\end{pro}
For $u \geq u_1$, 
\begin{equation*}
\dot{v} = -\sqrt{x_0p_x^2(0)(e^{2v(u_1)} - e^{2v(u)})},
\end{equation*}
and we are led to:
\begin{align*}
 \int^g \frac{dy}{\sqrt{x_0p_x^2(0)(e^{2v(u_1)} - e^{2y})}}, \quad g < v(u_1).
\end{align*}
Thus, on this interval, we get: 
\begin{pro}
Assume $\dot{v}(0) > 0$. Then, for any $u \geq u_1$, 
\begin{equation*}
|\lambda(u)| =  \frac{2|\lambda(u_1)|e^{\sqrt{C_3}(u-u_1)}}{1+e^{2\sqrt{C_3}(u-u_1)}},
\end{equation*}
where
\begin{equation}\label{Constantes2}
C_3:= x_0p_x^2(0)e^{2v(u_1)} = x_0p_x^2(0)|\lambda(u_1)|^2 = x_0p_x^2(0)|\lambda_0|^2 + \dot{v}(0)^2 = C_1. 
\end{equation}
\end{pro}
\end{itemize}
\begin{rem}\label{x0negative}
Note that $p_x(0) > 0$ by the faithfulness of $\tau$. Although we are interested in the case $x_0>0$, the following two cases play a key role in the proof of the main result of the paper. 
\begin{itemize}
\item If  we let $x_0p_x^2(0) \rightarrow 0^+$, then $u_1$ tends to infinity locally uniformly in $\lambda_0$ and the first two expressions of $|\lambda(u)|$ reduce to:  
\begin{equation*}
|\lambda(u)| = e^{\dot{v}(0)u+v_0} = |\lambda_0|e^{\dot{v}(0)u}, 
\end{equation*}
giving the solution to the equation $\ddot{v} = 0$.
\item If $x_0$ is negative so that $p_x(0)>0$, we easily see from \eqref{px} that $p_x$ remains positive until it blows up and therefore,  using similar computations, we show that the system \eqref{sys} admits a solution up to the blow-up time.
\end{itemize}
\end{rem}

As to the angular part of the curve $u \mapsto \lambda(u)$, it admits the following expression: 
\begin{pro}\label{argument}
As long as the solutions of \eqref{sys} exist, we have: 
\begin{equation*}
\theta(u) = \theta(0)+ K_2u/2,
\end{equation*}
where we recall $K_2 = a_0p_b(0) - b_0p_a(0)$ is the angular momentum.
\end{pro}
\begin{proof}
Since $\theta$ is a linear map, then
\begin{align*}
\dot{\theta}(0) = \dot{\theta}(u) & = \frac{d}{du}\arctan\left(\frac{b}{a}\right)(u) 
\\& = \frac{\dot{b}a-\dot{a}b}{a^2+b^2}(u)
\\&=\frac{[(a^2-b^2)(ap_b+bp_a) - 2ab(ap_a - bp_b)] }{2(a^2+b^2)}(u)
\\&=\frac{ap_b - bp_a}{2}(u).
\end{align*}
Since $ap_b - bp_a$ is a constant of motion, the proposition follows. 
\end{proof}

\subsection{Proof of Theorem \ref{expression}}
In this paragraph, we proceed to the  proof of Theorem \ref{expression} which provides an expression of $S$ along characteristic curves. 
\begin{proof}[Proof of Theorem \ref{expression}]
Set $J(u) :=(\lambda(u),x(u)) = (a(u), b(u), x(u))$ and $P(u):=(p_a(u),p_b(u),p_x(u))$. Then, we compute the inner product:
\begin{align*}
    P \cdot \frac{dJ}{du}&=P \cdot \nabla_PH =2H-\frac{a}{2}p_a-\frac{b}{2}p_b =2H_0-\frac{a}{2}p_a-\frac{b}{2}p_b.
\end{align*}
whence (see e.g. \cite{DHK}, Proposition 6.3),
\begin{align*}
S(u,J(u)) & =S(0,J_0)+H_0u-\frac{1}{2}\int_0^u[a(s)p_a(s)+b(s)p_b(s)]ds
\\& = \tau(\log(|h-\lambda_0|^2+x_0)) + H_0 u -\frac{1}{2}\int_0^u[a(s)p_a(s)+b(s)p_b(s)]ds.
\end{align*}
Moreover, using Proposition \ref{ConsMot} together with 
\begin{equation*}
x(s)p_x(s) = x_0p_x(0) \left(1-p_x(0)\int_0^u|\lambda(s)|^2ds\right),
\end{equation*}
it follows that:
\begin{multline*}
 \frac{1}{2}\int_0^ua(s)p_a(s)+b(s)p_b(s)ds =\left(x_0p_x(0)+\frac{1}{2}a_0p_a(0)+\frac{1}{2}b_0p_b(0)\right)u\\ - \int_0^u x(s)p_x(s)ds 
 = \left(x_0p_x(0)+\frac{1}{2}a_0p_a(0)+\frac{1}{2}b_0p_b(0)\right)u - x_0p_x(0)u
  \\ +x_0p_x^2(0)\int_0^u\int_0^s|\lambda(y)|^2dy =\left(\frac{1}{2}a_0p_a(0)+\frac{1}{2}b_0p_b(0)\right)u +x_0p_x^2(0)\int_0^u\int_0^s|\lambda(y)|^2dy.
\end{multline*}
Finally, recall that $|\lambda|$ satisfies
\begin{align*}
 x_0p_x^2(0) |\lambda|^2 + \ddot{(\ln(|\lambda|))}  =  0,
\end{align*}
which implies that
\begin{align*}
x_0p_x^2(0)\int_0^u\int_0^s|\lambda(y)|^2dy &=-\ln|\lambda(u)|+\ln|\lambda_0| + \dot{v}(0) u
\\& = -\ln|\lambda(u)|+\ln|\lambda_0| + \frac{1}{2}\left[1- a_0p_a(0)-b_0p_b(0)\right]u.
\end{align*}
The sought expression follows then after straightforward computations.
\end{proof}

\subsection{Blow-up time} 
We close this section with the explicit formula for the time 
\begin{equation*}
t_* := t_*(\lambda_0, x_0)
\end{equation*} 
defined by: 
\begin{equation*}
\int_0^{t_*}|\lambda(s)|^2ds=\frac{1}{p_x(0)} =\frac{1}{\tau(q_0)}.
\end{equation*}
which is the first time when $p_x$ blows-up or equivalently the curve $t \mapsto x(t)$ attains zero. In order to unify both cases corresponding to $\dot{v}(0) > 0$ and $\dot{v}(0) \leq 0$, we shall use the left-continuous sign function:  
\begin{equation*}
\textrm{sgn}(u)=\begin{cases}
\displaystyle 1; \quad u >  0\\
-1; \quad u \leq 0
\end{cases},
\end{equation*}
and assume that $x_0p_x^2(0)$ is small enough so that $t_1 \geq t_*$. 
\begin{pro}
The blow-up time $t_*$ is given by
\begin{multline*}
\frac{2|\lambda_0|^2}{C_2|\lambda_0|^2 +e^{-2\textrm{sgn}(\dot{v}(0))\sqrt{C_1}t_*}(1+\sqrt{1-C_2|\lambda_0|^2})^2} = \\ \frac{2|\lambda_0|^2}{C_2|\lambda_0|^2 +(1+\sqrt{1-C_2|\lambda_0|^2})^2} + \frac{\sqrt{C_1}}{\tau(q_0)}\textrm{sgn}(\dot{v}(0)).
\end{multline*}
In particular, if $x_0p_x^2(0) \rightarrow 0^+$ then 
\begin{equation*}
t_* = \frac{1}{2\dot{v}(0)}\ln\left( 1+\frac{2\dot{v}(0)}{|\lambda_0|^2\tau(q_0)} \right).
\end{equation*}
\end{pro}
\begin{proof}
Recall that if $x_0p_x^2(0)$ is sufficiently small, $|\lambda(u)|$ takes the form:
\begin{equation*}
|\lambda(u)| =  \frac{2|\lambda_0|e^{-\textrm{sgn}(\dot{v}(0))\sqrt{C_1}u}(1+\sqrt{1-C_2|\lambda_0|^2})}{C_2|\lambda_0|^2 +e^{-2\textrm{sgn}(\dot{v}(0))\sqrt{C_1}u}(1+\sqrt{1-C_2|\lambda_0|^2})^2}.
\end{equation*}
Then 
\begin{align*}
\int_0^u|\lambda(s)|^2ds = \int_0^u \frac{4|\lambda_0|^2e^{-2\textrm{sgn}(\dot{v}(0))\sqrt{C_1}s}(1+\sqrt{1-C_2|\lambda_0|^2})^2}{[C_2|\lambda_0|^2 +e^{-2\textrm{sgn}(\dot{v}(0))\sqrt{C_1}s}(1+\sqrt{1-C_2|\lambda_0|^2})^2]^2}ds.
\end{align*}
Performing the variable change: 
\begin{equation*}
y = e^{-2\textrm{sgn}(\dot{v}(0))\sqrt{C_1}s}, 
\end{equation*}
then the equation satisfied by $t_*$ and its limit as $x_0p_x^2(0) \rightarrow 0^+$ follow from straightforward computations. 
\end{proof}
\begin{rem}
The equation satisfied by $t_*$ may be rewritten as: 
\begin{equation*}
|\lambda(t_*)| e^{\textrm{sgn}(\dot{v}(0))\sqrt{C_1}t_*} = |\lambda_0|+ \frac{\sqrt{C_1}(1+\sqrt{1-C_2|\lambda_0|^2})}{\tau(q_0)|\lambda_0|}\textrm{sgn}(\dot{v}(0)).
\end{equation*}
\end{rem}
Remember that in the Hamiltonian picture, $p_x$ is the partial derivative $\partial_xS$ along the characteristic curves $u \mapsto (\lambda(u), x(u))$. Therefore, if the curve $u \mapsto \lambda(u)$ starts in the resolvent set of $Y_{t_*}h$ then the identity \eqref{Id1} shows that $\lambda(t_*)$ lies in the spectrum of $Y_{t_*}h$. This observation raises the following problem: given a time $t > 0$, find a curve $\lambda$ (i.e. initial conditions) such that $t_* = t$? As suggested by our previous result, such curves are easy to find whenever the limit $x_0p_x^2(0) \rightarrow 0^+$ is allowed. For instance, if $|h-\lambda_0|^2$ is invertible with bounded or integrable inverse then $p_x(0) < \infty$ for all $x_0 \geq 0$ and we can let $x_0 \rightarrow 0^+$ in which case the whole curve $u \mapsto x(u)$ vanishes. Otherwise, letting $x_0 \rightarrow 0^+$ yields $p_x(0) = +\infty$  in which case $t_* = 0$.

\section{The support of the Brown measure of $Y_tP$}
In this section, we deal with the special case $h = P$ where we recall that $P$ is a selfadjoint projection with rank $\tau(P) = \alpha \in (0,1)$. Define:

\begin{align*}
T_{\alpha}(\lambda_0):=t_*(\lambda_0,0) =
\begin{cases}
 \frac{1}{2\dot{v}(0)}\ln\left( 1+\frac{2\dot{v}(0)}{|\lambda_0|^2\tau(q_0)} \right), \quad \dot{v}(0) \neq 0, \lambda_0 \neq 1, \\ 
 \frac{1}{|\lambda_0|^2\tau(q_0)}, \quad \dot{v}(0) = 0, \lambda_0 \neq 1, \\ 
 0, \quad \lambda_0 =1
 \end{cases}.
\end{align*}
Since
\begin{align*}
2\dot{v}(0) = 2\alpha-1 + 2\alpha \frac{a_0- |\lambda_0|^2}{|\lambda_0-1|^2} & = - \frac{(a_0 - (1-\alpha))^2 + b_0^2 - \alpha^2}{|\lambda_0-1|^2} 
\\& = \frac{\alpha(1-|\lambda_0|^2)-(1-\alpha)|1-\lambda_0|^2}{|\lambda_0-1|^2},
\end{align*}
and 
\begin{equation*}
|\lambda_0|^2 \tau(q_0) = 1+\alpha\frac{2a_0-1}{|\lambda_0-1|^2} = \frac{\alpha|\lambda_0|^2 + (1-\alpha)|\lambda_0-1|^2}{|\lambda_0-1|^2}, 
\end{equation*}
it follows that:
\begin{align*}
T_{\alpha}(\lambda_0)=
\begin{cases}
\displaystyle \frac{|1-\lambda_0|^2}{\alpha(1-|\lambda_0|^2)- (1-\alpha)|1-\lambda_0|^2}\log\left(\frac{\alpha}{\alpha |\lambda_0|^2+ (1-\alpha)|1-\lambda_0|^2} \right),\, |1-\alpha - \lambda_0| \neq \alpha, \\
\displaystyle \frac{|\lambda_0-1|^2}{\alpha|\lambda_0|^2 + (1-\alpha)|\lambda_0-1|^2} = \frac{|\lambda_0-1|^2}{\alpha}, \, |1-\alpha - \lambda_0| = \alpha,
\end{cases}.
\end{align*}
In particular, $T_{\alpha}, \alpha\in(0,1),$ is a real analytic function on the whole plane. Note that this description of $T_{\alpha}(\lambda_0)$ extends continuously to $\alpha=1$ since:
\begin{align*}
T_{1}(\lambda_0)=
\begin{cases}
\displaystyle \frac{|1-\lambda_0|^2}{|\lambda_0|^2-1}\log\left(|\lambda_0|^2 \right), \quad |\lambda_0|\neq 1,\\
 |1-\lambda_0|^2, \quad |\lambda_0|=1,
\end{cases}.
\end{align*}
However, $T_1$ is real analytic on the punctured plane $\mathbb{R}^2\setminus\{0\}$.

\subsection{The regions $\Sigma_{t,\alpha}$}
Recall the map:
\begin{align*}
f_{t,\alpha}(z) := ze^{\frac{t}{2}\frac{2\alpha-1+z}{1-z}}, \quad t\ge0, z\in \mathbb{C}\setminus\{1\}.
\end{align*}
We can easily see that the equality
\begin{equation*}
 |f_{t,\alpha}(z)|^2 =\frac{\alpha|z|^2}{\alpha|z|^2+(1-\alpha)|1-z|^2}
\end{equation*}
is satisfied on the circle $\mathbb{T}(1-\alpha,\alpha)$ with center $1-\alpha$ and radius $\alpha$  (we take a tangential limit at $z=1$). However, it is also satisfied outside this circle and we are led to consider the following set:
\begin{equation*}
G_{t,\alpha} := \left\{z \in \mathbb{C}: |1-\alpha-z|\ne \alpha, \quad |f_{t,\alpha}(z)|^2 =\frac{\alpha|z|^2}{\alpha|z|^2+(1-\alpha)|1-z|^2} \right\}. 
\end{equation*}
In this respect, recall that $F_{t,\alpha}$ is the closure of $G_{t,\alpha}$. Then

\begin{lem}\label{Jordan}
Let $\alpha \in (0,1)$.  Then, $F_{t,\alpha}$ is a Jordan curve  for each $t>0$. Moreover, $f_{t,\alpha}$ is a one-to-one map there. 
\end{lem}
\begin{proof}
Equivalently, we shall consider the image of  $G_{t,\alpha}$ under the Mobius transformation: 
\begin{equation*}
z\mapsto w = \frac{2\alpha-1+z}{1-z}
\end{equation*}
since clearly $G_{t,\alpha}$ does not contain $z=1$. Note also that the circle 
\begin{equation*}
\mathbb{T}(1-\alpha,\alpha) = \{z, |1-\alpha-z| = \alpha\}
\end{equation*}
is mapped onto $\{\Re(w) = 0\} \cup \{\infty\}$. Now, write $w =x+iy$ then
\begin{align*}
 |f_{t,\alpha}(z)|^2 =\frac{\alpha|z|^2}{\alpha|z|^2+(1-\alpha)|1-z|^2} &\Leftrightarrow \left|\frac{x+1-2\alpha+iy}{x+1+iy}\right|^2e^{tx}=\frac{\alpha(x+1-2\alpha)^2+\alpha y^2}{\alpha(x+1-2\alpha)^2+\alpha y^2+4\alpha^2(1-\alpha)}
  \\&\Leftrightarrow \frac{x^2+2x(1-2\alpha)+1+y^2}{(x+1)^2+y^2} e^{tx}=1.
\end{align*}
Consequently, 
\begin{equation}
y^2 = \frac{\phi_{t,\alpha}(x)}{e^{tx}-1}, \quad x \neq 0, 
\end{equation}
where
\begin{equation*}
\phi_{t,\alpha}(x)=(x+1)^2-(x^2+2x(1-2\alpha)+1)e^{tx},
\end{equation*}
while 
\begin{equation*}
y^2 =  \frac{4\alpha}{t}-1; \quad x=0,
\end{equation*}
provided that $t \leq 4\alpha$. The map $\phi_{t,\alpha}$ is smooth and satisfies $\phi_{t,\alpha}(0)=0, \phi_{t,\alpha}(-1)=-4\alpha e^{-t}<0$ together with the limits: 
\begin{equation*}
\lim_{x\rightarrow-\infty}\phi_{t,\alpha}(x)= + \infty,\quad \lim_{x\rightarrow+\infty}\phi_{t,\alpha}(x)=-\infty.
\end{equation*}
Moreover its zero set coincides with the number of the roots of the equation:
\begin{equation*}
1-e^{-tx} = \frac{4\alpha x}{(x+1)^2}. 
\end{equation*} 
A quick inspection shows then that there is a unique negative root $x_{t,\alpha}^- < -1$ for which $\phi_{t,\alpha}(x) \geq 0$ on $(-\infty, x_{t,\alpha}^-]$. Moreover, there is at most one positive solution and at most another negative solution in $(-1,0)$. In particular, there is no negative solution in $(-1,0)$ when $t \leq 4\alpha$. To see this last fact, we note that for any $x \in (-1,0)$, 
\begin{equation*}
e^{-t x} - 1  + \frac{4\alpha x}{(1+x)^2} \leq e^{-t x} - 1  + \frac{t x}{(1+x)^2} . 
\end{equation*}
Then it suffices to prove that 
\begin{equation}\label{IneqZero}
e^{-t x} - 1  + \frac{t x}{(1+x)^2} < 0, \quad  x \in (-1,0). 
\end{equation}
To this end, we differentiate the LHS of this inequality with respect to $t \in (0,4)$ to get: 
 \begin{equation*}
\frac{-x[(1+x)e^{-t x/2} +1][(1+x)e^{-t x/2} -1]}{(1+x)^2}. 
\end{equation*}
But, the variations of the real function: 
\begin{equation*}
x \mapsto (1+x)e^{-t x/2} -1
\end{equation*}
on the interval $(-1,0)$ show that it is negative when $t < 2$ while its sign changes exactly once from negative to positive when $t > 2$. If $t < 2$, then the inequality \eqref{IneqZero} is clear since the function
\begin{equation*}
t\mapsto e^{-t x} - 1  + \frac{t x}{(1+x)^2}  
\end{equation*}
is decreasing and vanishes at $t =0$. Otherwise, by continuity at $t = 2$, it only remains to prove that the value of 
 \begin{equation*}
e^{-t x} - 1 + \frac{t x}{(1+x)^2}   
\end{equation*}
at $t =4$ is negative. Explicitly,
\begin{equation*}
e^{-4 x} - 1 + \frac{4 x}{(1+x)^2}=\left(e^{-2x} - \frac{1-x}{1+x}\right)\left(e^{-2x} + \frac{1-x}{1+x}\right)  < 0 , \quad x \in (-1,0),
\end{equation*}
or equivalently: 
\begin{equation}\label{IneqZero1}
e^{-2x} - \frac{1-x}{1+x} < 0 , \quad x \in (-1,0). 
\end{equation} 
Since the derivative of the LHS of this inequality is given by: 
\begin{equation*}
2\frac{[1+(1+x)e^{-x}][1-(1+x)e^{-x}]}{(1+x)^2},
\end{equation*}
and since $1-(1+x)e^{-x} > 0$ on $(-1,0)$ then \eqref{IneqZero1} holds true. 
Now, since $\phi_{t,\alpha}'(0)=4\alpha-t$, then we shall then distinguish separately the three cases $t<4\alpha,\ t=4\alpha$ and $t>4\alpha$. 

\begin{itemize}
\item $t<4\alpha$: in this case $\phi_{t,\alpha}'(0)>0$ therefore there exists a unique $x_{t,\alpha}^+>0$ such that $\phi_{t,\alpha}(x) \geq 0$ on $[0,x_{t,\alpha}^+]$. 
Consequently, the image of $F_{t,\alpha}$ under the Mobius transformation above is parametrized by:   
\begin{equation*}
x \in [x_{t,\alpha}^-, x_{t,\alpha}^+] \setminus \{0\}, \quad y_{t,\alpha}(x) = \pm \sqrt{\frac{\phi_{t,\alpha}(x)}{e^{tx}-1}},
\end{equation*}
and  
\begin{equation*}
x =0, \quad y_{t,\alpha}(0) \in\left\{ \pm \sqrt{\frac{4\alpha}{t}-1}\right\},
\end{equation*}
 which is clearly a Jordan curve. 
\item $t = 4\alpha$: since $\phi_{t,\alpha}'(0) = 0$ then $\phi_{t,\alpha}(x) \leq 0$ for all $x \geq 0$. As a matter of fact, the parametrization of the image of $F_{t,\alpha}$ is given by: 
\begin{equation*}
x \in [x_{t,\alpha}^-, 0) , \quad y_{t,\alpha}(x) = \pm \sqrt{\frac{\phi_{t,\alpha}(x)}{e^{tx}-1}},
\end{equation*}
and  $x =0, \ y_{t,\alpha}(0) =0$, therefore the image of $F_{t,\alpha}$ is a Jordan curve as well. 
\item $t > 4\alpha$: this case is similar to the previous one since $\phi_{t,\alpha}'(0) < 0$ which forces the existence of a negative root $\tilde{x}_{t,\alpha}^-$ in $(-1,0)$. We then have $\phi_{t,\alpha}(x) > 0$ on $(\tilde{x}_{t,\alpha}^-, 0)$ and $\phi_{t,\alpha}(x) \leq 0$ for $x \geq 0$. The corresponding parametrization is then given by: 
\begin{equation*}
x \in [x_{t,\alpha}^-, \tilde{x}_{t,\alpha}^-] , \quad y_{t,\alpha}(x) = \pm \sqrt{\frac{\phi_{t,\alpha}(x)}{e^{tx}-1}},
\end{equation*}
which clearly yields a Jordan curve. 
\end{itemize}
Finally, take $z_1 \neq z_2$ lying on $F_{t,\alpha}$ such that $f_{t,\alpha}(z_1) = f_{t,\alpha}(z_2)$. Then 
\begin{equation*}
\left|\frac{1-z_1}{z_1}\right| = \left|\frac{1-z_2}{z_2}\right|. 
\end{equation*}
Setting: 
\begin{equation*}
w_1 = \frac{2\alpha-1+z_1}{1-z_1}, \quad w_2 = \frac{2\alpha-1+z_2}{1-z_2},
\end{equation*}
we get the equivalent condition: 
\begin{equation*}
|w_1+1-2\alpha| = |w_2+1-2\alpha|. 
\end{equation*}
Write $w_1 = x_1 + iy_1, w_2 = x_2+iy_2$ then the last condition reads: 
\begin{equation*}
(x_1+1-2\alpha)^2 + \frac{\phi_{t,\alpha}(x_1)}{e^{tx_1}-1} = (x_2+1-2\alpha)^2 + \frac{\phi_{t,\alpha}(x_2)}{e^{tx_2}-1}
\end{equation*}
which reduces after some simplifications to: 
\begin{equation*}
\frac{e^{tx_1}-1}{x_1} = \frac{e^{tx_2}-1}{x_2}.
\end{equation*}
Since the map $u \mapsto (e^{tu}-1)/u$ is invertible on $\mathbb{R}$ then $x_1 = x_2$ and in turn $z_2 = z_1$ or $z_2 = \overline{z_1}$. The second alternative is only possible when $z_1$ is real since $f_{t,\alpha}(\overline{z_1}) = \overline{f_{t,\alpha}(z_1)}$, 
in which case $z_1 = z_2$ as well.   
\end{proof}

Let $\Sigma_{t,\alpha}$ be the bounded component of the complement of $F_{t,\alpha}$. By the Jordan curve Theorem, we have:
\begin{equation*}
\partial\Sigma_{t,\alpha}=F_{t,\alpha}.
\end{equation*}

\begin{figure}
\begin{center}
\includegraphics[width=5cm]{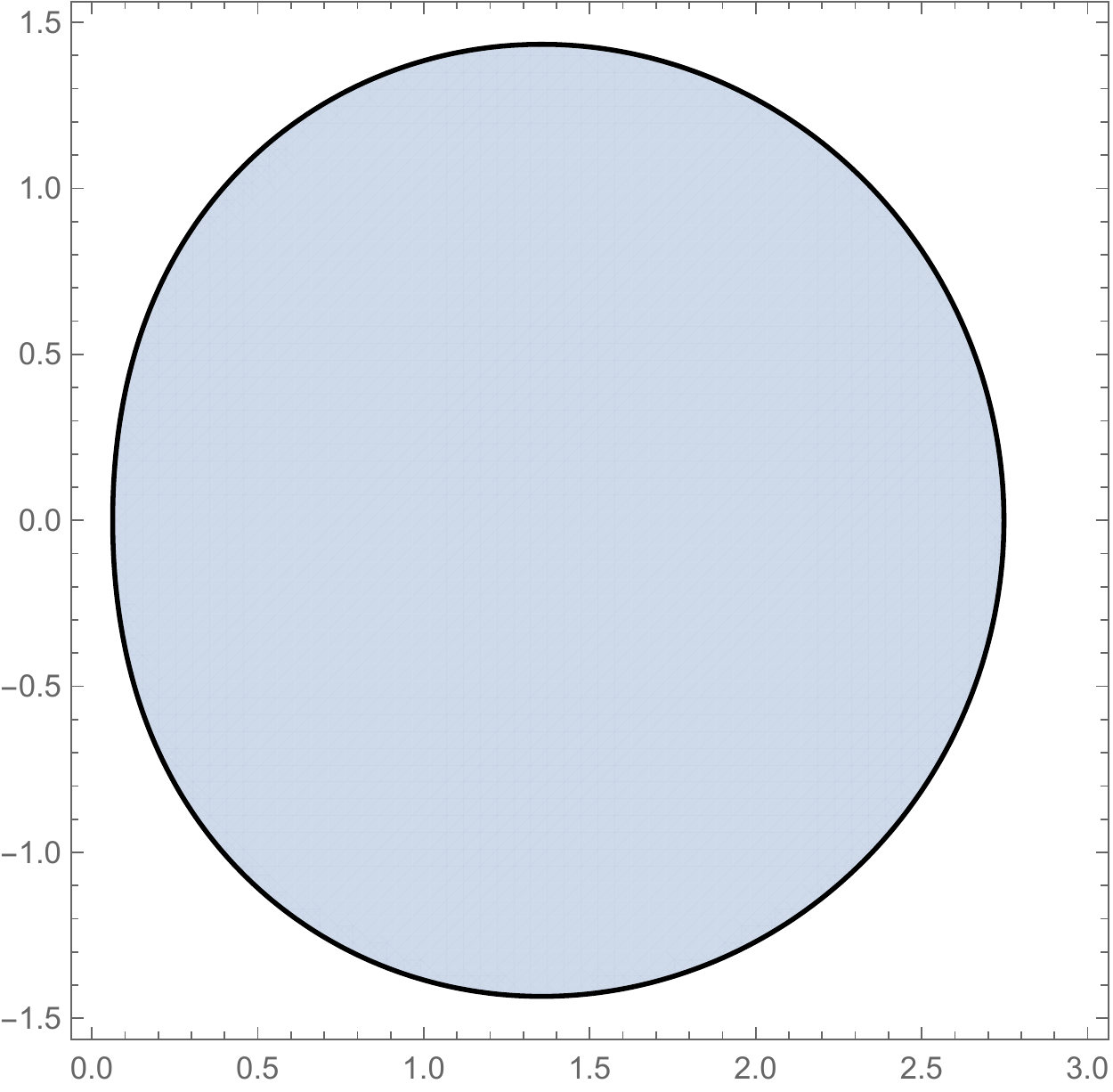} 
\includegraphics[width=5cm]{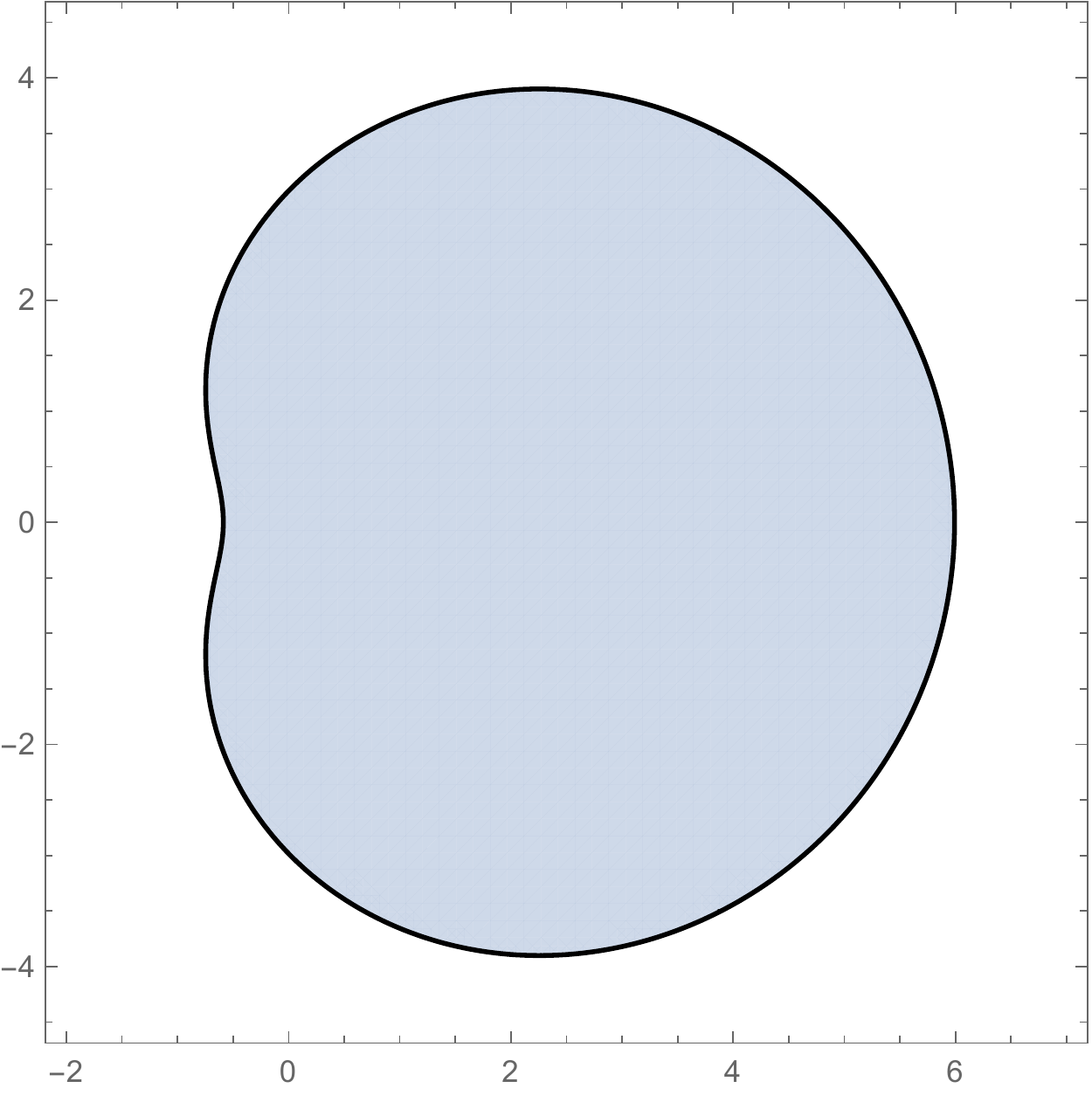} 
\end{center}
\caption{The regions $\Sigma_{t,\alpha}$ with $(t,\alpha)=(2,0.3)$ (left) and $(t,\alpha)=(3,0.7)$ (right)}.
\begin{center}
\includegraphics[width=5cm]{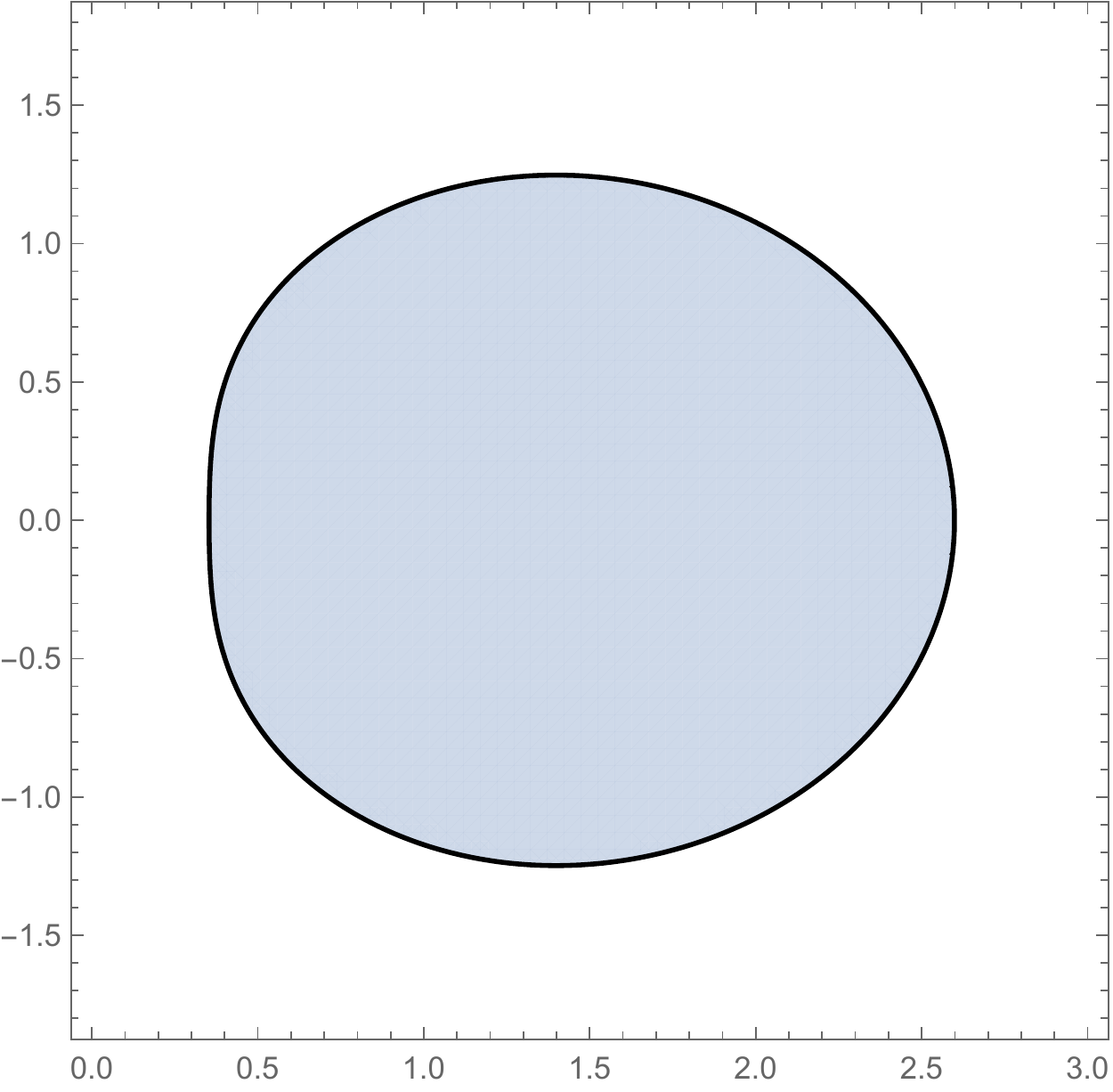} 
\includegraphics[width=5cm]{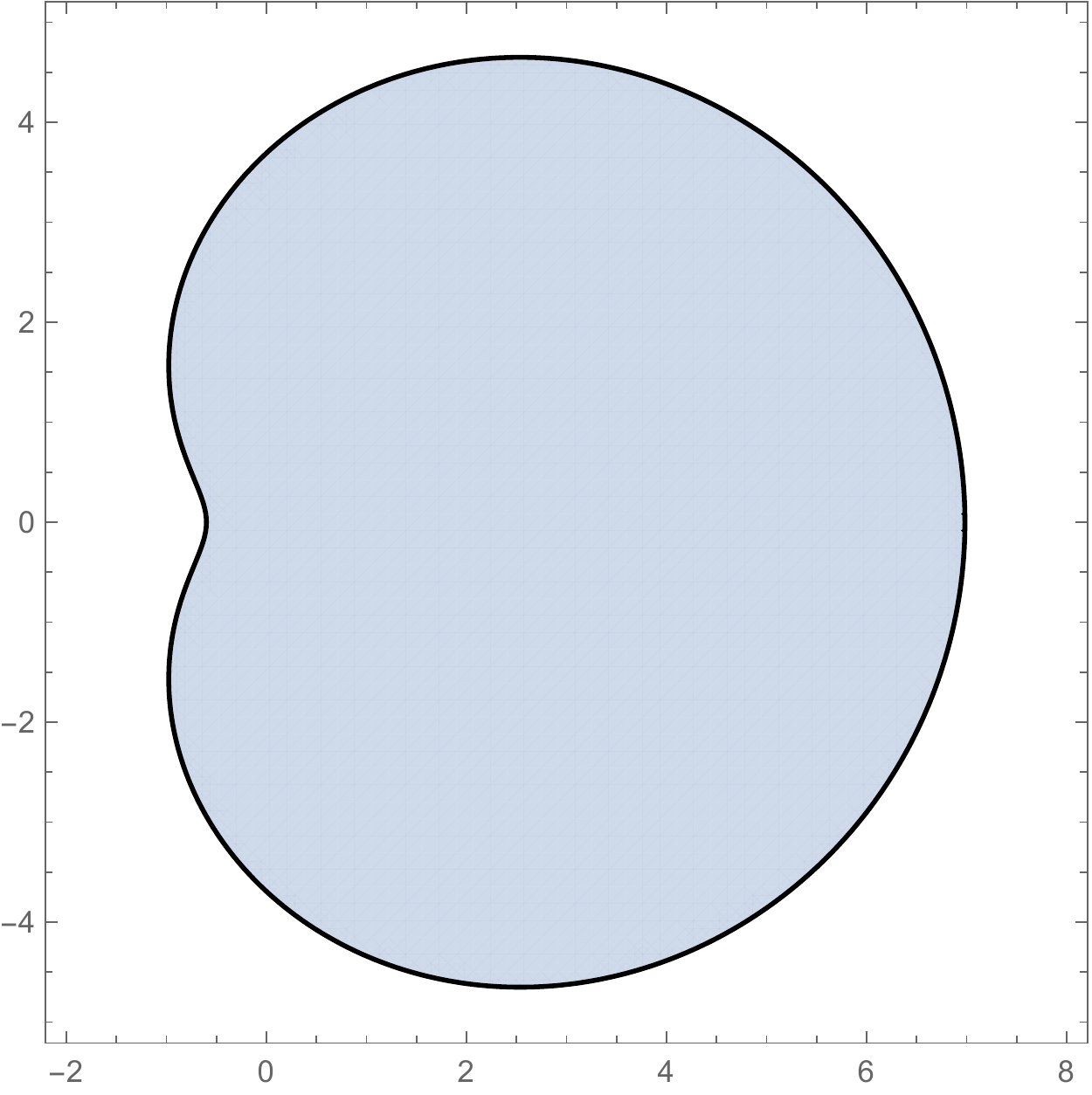} 
\includegraphics[width=5cm]{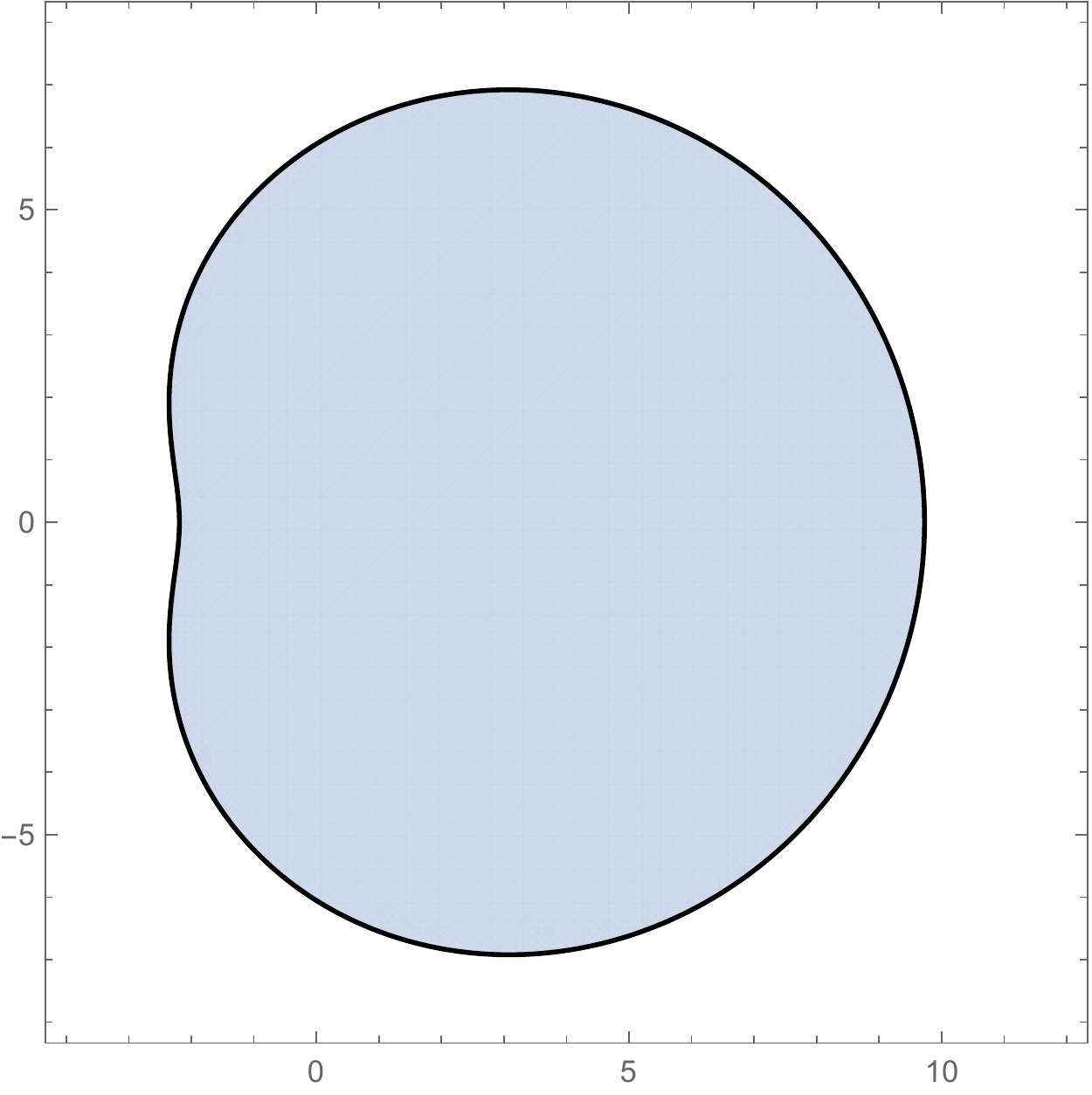} 
\includegraphics[width=5cm]{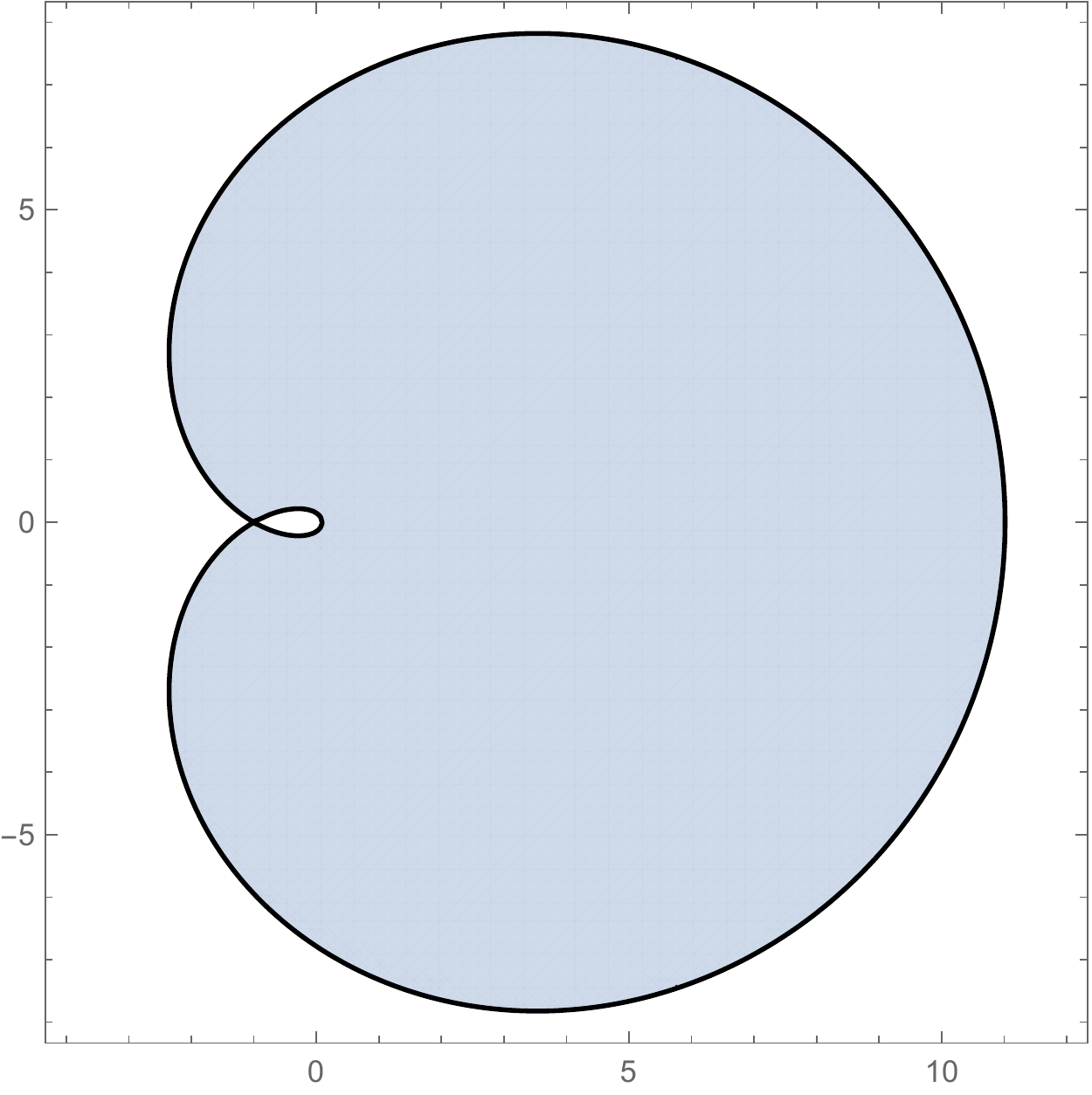} 
\end{center}
\caption{The regions $\Sigma_{t,\alpha}$ with $(t,\alpha)=(1,0.8)$ (up-left), $(3.2,0.8)$ (up-right), $(4,0.8)$ (down-left), $(4,1)$ (down-right)}.
\end{figure}

Then, $\Sigma_{t,\alpha}$ may be characterized as follows. 
\begin{pro}\label{boundary}
For any $t > 0$  and $\alpha\in (0,1)$, we have
\begin{equation*}
\Sigma_{t,\alpha}:=\{\lambda_0 \in \mathbb{C}: T_{\alpha}(\lambda_0) < t\}.
\end{equation*}
In particular, $1 \in \Sigma_{t,\alpha}$. 
\end{pro}
\begin{proof}
We note that when $x_0 = 0$, we have
\begin{align*}
|\lambda(t)|=|\lambda_0|e^{\dot{v}(0)t}=|\lambda_0|e^{-\frac{t}{2}-\alpha t\frac{(a_0-1)}{|\lambda_0-1|^2}},\quad \quad  \theta(t) = \alpha t \frac{b_0}{|\lambda_0-1|^2} + \theta_0.
\end{align*}
Equivalently 
\begin{equation*}
\lambda(t)= \lambda_0e^{\frac{t}{2}(-1+2\alpha/(1-\lambda_0))}   = f_{t,\alpha }(\lambda_0).
\end{equation*}
On the other hand, since
\begin{align*}
2\dot{v}(0) & = \frac{\alpha(1-|\lambda_0|^2)-(1-\alpha)|1-\lambda_0|^2}{|\lambda_0-1|^2},
\end{align*}
then the circle $\mathbb{T}(1-\alpha, \alpha)$ is exactly the zero set of $\dot{v}(0)$. Hence, if $\lambda_0$ is such that $\dot{v}(0)\ne 0$, then
\begin{multline*}
  |f_{t,\alpha}(\lambda_0)|^2 = |\lambda_0|^2e^{2t\dot{v}(0)}= \frac{\alpha|\lambda_0|^2}{\alpha|\lambda_0|^2+(1-\alpha)|1-\lambda_0|^2}  \\ \Leftrightarrow t =  \frac{|\lambda_0-1|^2}{\alpha(1-|\lambda_0|^2)-(1-\alpha)|1-\lambda_0|^2} \log\left( \frac{\alpha}{\alpha|\lambda_0|^2+(1-\alpha)|1-\lambda_0|^2}\right)
 \end{multline*}
Thus, the set $G_{t,\alpha}$, which coincides with the Jordan curve $F_{t,\alpha}$ except possibly at one or two points, is exactly the set where $T_{\alpha}(\lambda_0) = t$. But, the set
\begin{equation*}
\{\lambda_0, T_{\alpha}(\lambda_0) > t\},
\end{equation*}
is unbounded since 
\begin{equation*}
\lim_{|\lambda_0| \rightarrow +\infty} T_{\alpha}(\lambda_0) = +\infty
\end{equation*}
while \begin{equation*}
\{\lambda_0, T_{\alpha}(\lambda_0) < t\},
\end{equation*}
is bounded, then the first statement of the proposition is clear. As to the second one, it follows from the first statement together with $T_{\alpha}(1) = 0$.  
\end{proof}

\begin{rem}\label{homeomorphism}
When $\alpha = 1$, the map 
\begin{align*}
f_{t,1}(z)=ze^{\frac{t}{2}\frac{1+z}{1-z}}
\end{align*}
describes the spectrum of $Y_t$ (\cite{Biane1}) and encodes the support $\Sigma_{t,1}$ of the Brown measure of the free multiplicative Brownian motion (\cite{DHK}). Writing 
\begin{equation*}
f_{t,\alpha}(z) = e^{(\alpha-1)t/2}f_{\alpha t,1}(z)
\end{equation*}
we readily deduce from \cite{Biane1} (see paragraph 4.2.3) that $f_{t,\alpha}$ is a one-to-one map from the Jordan domain 
\begin{equation*}
\Gamma_{t,\alpha }:=\{z \in \mathbb{D},  |f_{t,\alpha}(z)| < e^{(\alpha-1)t/2}\} = \{z \in \mathbb{D},  |f_{\alpha t,1}(z)| <1\}
\end{equation*} 
onto the open disc $\mathbb{D}(0,e^{(\alpha-1)t/2})$ and from a neighborhood of infinity (the image of $\Gamma_{t,\alpha }$ under the inversion $z \mapsto 1/z$) onto $\mathbb{C} \setminus \overline{\mathbb{D}(0,e^{(\alpha-1)t/2})}$. In both cases, it extends to a homeomorphism between the corresponding boundaries. These properties satisfied by $f_{t,\alpha}$ will be used to prove Theorem \ref{surj-outside} below. Note also that $\Sigma_{t,1}$ may be doubly-connected in contrast to $\Sigma_{t,\alpha}, \alpha \in (0,1)$ (see \cite{Biane1}). 
\end{rem}

\subsection{Proof of Theorem \ref{support}}

From Lemma \ref{Jordan}, we deduce that $ f_{t,\alpha}\big(F_{t,\alpha} \big)$ is a Jordan curve for any $t>0$  and any $\alpha\in (0,1)$. Denote $\Omega_{t,\alpha}$ the region enclosed by this curve.
\begin{figure}
\begin{center}
\includegraphics[width=4.8cm]{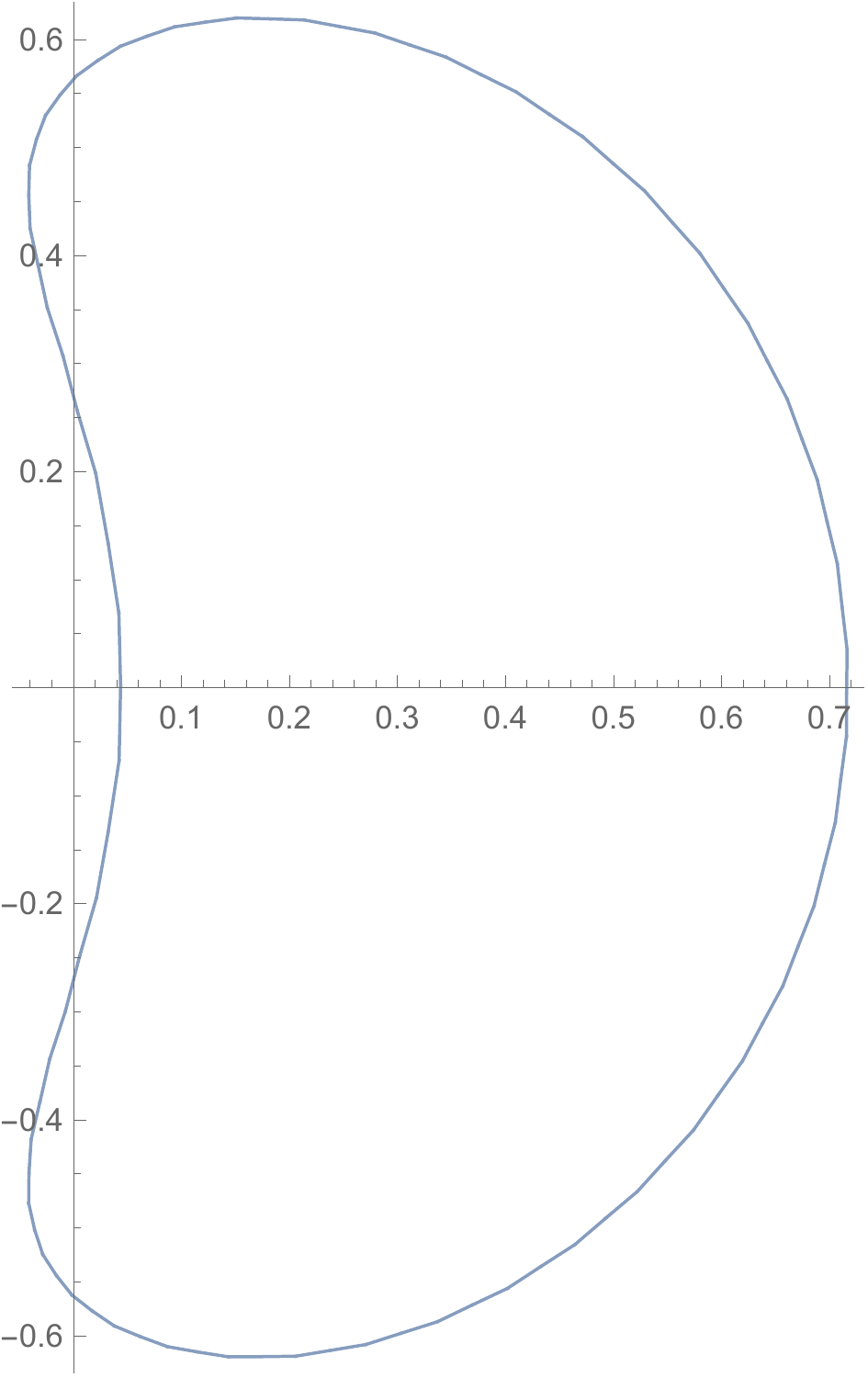} 
\includegraphics[width=7cm]{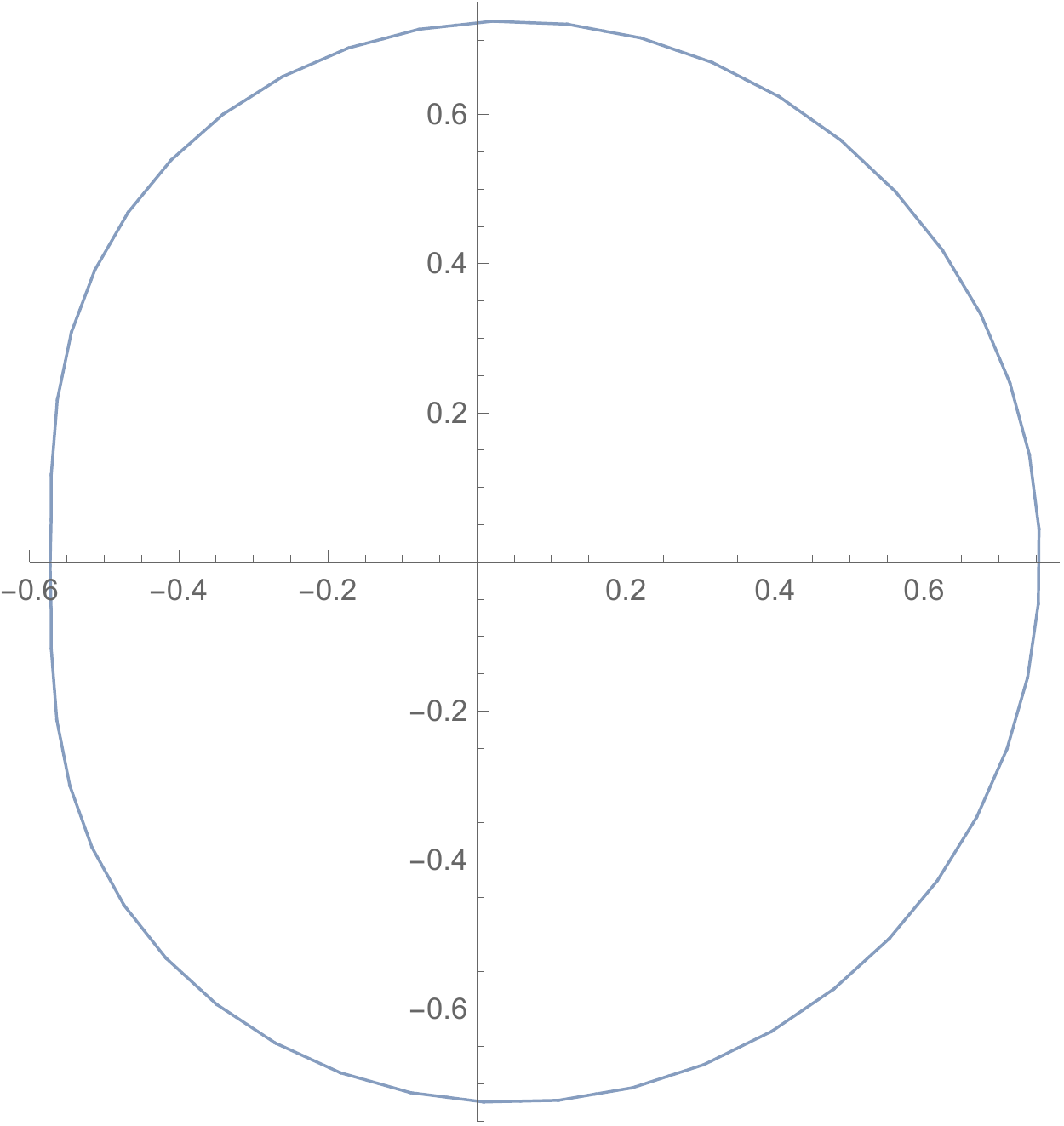} 
\end{center}
\caption{The regions $\Omega_{t,\alpha}$ with $(t,\alpha)=(2,0.3)$ (left) and $(t,\alpha)=(4,0.5)$ (right)}.
\end{figure}
The major step toward proving Theorem \ref{support} is the following Theorem which shows that any complex number lying outside $\overline{\Omega}_{t,\alpha}$ may be reached by a characteristic curve $u \mapsto \lambda(u)$ exactly at time $t$. 
\begin{teo}\label{surj-outside}
Let $\lambda$ outside $\overline{\Omega}_{t,\alpha}, t > 0, \alpha \in (0,1)$. Then there exists  $\lambda_0$ outside $\overline{\Sigma}_{t,\alpha}$ so that the solution to the system \eqref{sys} is well defined up to time $t$ and
 $\lambda(u)$ approaches $\lambda$ as $u$ approaches $t^-$. Moreover, 
\begin{equation*}
\lambda_0 = f_{t,\alpha}^{-1}(\lambda) \in \Gamma_{t,\alpha} \cup 1/\Gamma_{t,\alpha}
\end{equation*}
provided that $|\lambda| \neq e^{(\alpha-1)t/2}$ in which case it is unique. 
\end{teo}
We will need the following elementary lemma.
\begin{lem}\label{deriv}
For any $t>0$, the map
\begin{equation*}
g_{t}:\alpha \mapsto   2e^t-(1+(1-\alpha) t)e^{\alpha t}
\end{equation*}
is positive on $[0, 1]$.
\end{lem}
\begin{proof}
We have
\begin{equation*}
g_t'(\alpha)=-(1-\alpha)t^2e^{\alpha t}\le0.
\end{equation*}
Then, $g_t(\alpha)\ge g_t(1)=e^t> 0.$
\end{proof}
From this one can deduce.
\begin{lem}\label{numerator}
For any $t>0$ and $\alpha \in (0,1)$, the map
\begin{equation*}
\psi_{t,\alpha}:x \mapsto x-(x+1-\alpha)e^{2t(1-\alpha)} + (1-\alpha)e^{t(x+2(1-\alpha))}
\end{equation*}
is positive on $[2(\alpha-1), \alpha-1]$.
\end{lem}
\begin{proof}
Differentiating twice the function $\psi_{t,\alpha}$, we get
\begin{equation*}
\psi_{t,\alpha}''(x)=(1-\alpha)t^2e^{ t(x+2(1-\alpha))}>0.
\end{equation*}
Then, $\psi_{t,\alpha}'$ is increasing on $[2(\alpha-1), \alpha-1]$ and in particular,
\begin{equation*}
\psi_{t,\alpha}'(x)\le\psi_{t,\alpha}'(\alpha-1)
\end{equation*}
for any $x\in[2(\alpha-1), \alpha-1]$.
Next, differentiating $\psi_{t,\alpha}'(\alpha-1)$ with respect to $\alpha$, we readily get by Lemma \ref{deriv}
\begin{equation*}
\partial_\alpha \psi_{t,\alpha}'(\alpha-1)=tg_t(\alpha)e^{t(1-2\alpha)}>0.
\end{equation*}
Thus, for any $\alpha\in(0,1)$, we have 
\begin{equation*}
\psi_{t,\alpha}'(x)\le\psi_{t,\alpha}'(\alpha-1)<\psi_{t,1}'(0)=0.
\end{equation*}
and then $\psi_{t,\alpha}$ is decreasing on  $[2(\alpha-1), \alpha-1]$.
Hence, 
\begin{equation*}
\psi_{t,\alpha}(x)\ge\psi_{t,\alpha}(\alpha-1)=(1-\alpha)e^{-\alpha t}(e^{ t}-e^{\alpha t})>0,
\end{equation*}
for any $x\in[2(\alpha-1), \alpha-1]$.
\end{proof}

We are now ready for the proof our surjectivity result.
\begin{proof}[Proof of Theorem \ref{surj-outside}]
Fix $\lambda \notin \overline{\Omega}_{t,\alpha}$. Then, the second statement of the proposition follows readily from Remark \ref{homeomorphism}. Moreover, if $|\lambda| = e^{(\alpha-1)t/2}$ then the same remark shows that there are two inverse images of 
$\lambda$ by $f_{t,\alpha}$ and lying respectively on the boundaries of $\Gamma_{t,\alpha}$ and of $1/\Gamma_{t,\alpha}$. 
As a matter of fact, it only remains to prove that if $\lambda = f_{t,\alpha}(\lambda_0) \notin \overline{\Omega}_{t,\alpha}$ then $\lambda_0 \notin \overline{\Sigma}_{t,\alpha}$.

To this end, consider the circle $\mathcal{C}(\lambda)$ of radius $|\lambda|$ and centered at the origin. If $|\lambda| \leq e^{(\alpha-1)t/2}$ then $f_{t,\alpha}^{-1}\left(\mathcal{C}(\lambda)\right)$ is a Jordan curve 
$\mathcal{C}_{t,\alpha}(\lambda_0) \subset \overline{\Gamma}_{t,\alpha}$ which intersects $\partial\Sigma_{t,\alpha}$ in no more than two points. The last fact is readily seen by solving the equation: 
\begin{equation*}
|f_{t,\alpha}(z)| = \frac{\alpha |z|}{\alpha |z| + (1-\alpha) |1-z|} = |\lambda|, \quad z \in \partial\Sigma_{t,\alpha},
\end{equation*}
and by mimicking the proof of the injectivity of $f_{t,\alpha}$ on $\partial\Sigma_{t,\alpha}$. Moreover, since $f_{t,\alpha}$ is one-to-one on $\mathcal{C}_{t,\alpha}(\lambda_0)$ and on $\partial\Sigma_{t,\alpha}$ then 
\begin{equation*}
f_{t,\alpha}\left(\mathcal{C}_{t,\alpha}(\lambda_0) \cap \partial\Sigma_{t,\alpha}\right) = \mathcal{C}(\lambda) \cap \partial\Omega_{t,\alpha}, 
\end{equation*}
which shows that the sets $\mathcal{C}_{t,\alpha}(\lambda_0) \cap \partial\Sigma_{t,\alpha}$ and $\mathcal{C}(\lambda) \cap \partial\Omega_{t,\alpha}$ have the same cardinality. Another property we shall need below is that $f_{t,\alpha}^{-1}$ maps the interior of $\mathcal{C}(\lambda)$ onto the interior of $\mathcal{C}_{t,\alpha}(\lambda_0)$. This follows from the homeomorphism property which preserves simple-connectedness and from $f_{t,\alpha}(0) = 0$. 
\begin{itemize}
\item In the case of empty intersection, we have either  $\overline{\Omega}_{t,\alpha}$ lies entirely in the open disc  $\mathbb{D}(0,|\lambda|)$, and  this cannot happen since
\begin{equation*}
f_{t,\alpha}^{-1} (\partial\Omega_{t,\alpha}) \subset f_{t,\alpha}^{-1}\left(\mathbb{D}(0,|\lambda|)\right) \subset \Gamma_{t,\alpha} \subset \mathbb{D}(0,1),
\end{equation*}
then, $f_{t,\alpha}^{-1} (\partial\Omega_{t,\alpha}) \neq \partial \Sigma_{t,\alpha}$ since $z=1 \in \Sigma_{t,\alpha}$. 
Or  $\overline{\Omega}_{t,\alpha}\cap\overline{\mathbb{D}(0,|\lambda|)}=\emptyset$, 
then $\mathcal{C}_{t,\alpha}(\lambda_0)$ cannot lie inside $ \Sigma_{t,\alpha}$ since otherwise $0\in\Sigma_{t,\alpha}$ and then $0\in\Omega_{t,\alpha}$. Contradiction since   $\mathcal{C}(\lambda)$ lies outside $\overline{\Omega}_{t,\alpha}$.

\item The same arguments apply to the case of a single (necessary real) intersection point since the positions of the curves $\mathcal{C}(\lambda)$ and $\partial \Omega_{t,\alpha}$ are similar. 

\item Assume now that $\mathcal{C}(\lambda) \cap \partial \Omega_{t,\alpha}$ consists of two (complex-conjugate) distinct points. If $\partial \Omega_{t,\alpha}$ lies outside  $\mathbb{D}(0,|\lambda|)$, then $\lambda_0 \notin \overline{\Sigma}_{t,\alpha}^c$ by the same arguments used above. Otherwise,
$\partial \Omega_{t,\alpha}$ divides $\mathcal{C}(\lambda)$ into two Jordan domains. 
By discussing whether zero is inside or outside $\Omega_{t,\alpha}$ and using the fact that $f_{t,\alpha}^{-1}$ is a homeomorphism on $\mathcal{C}(\lambda)$ with $f_{t,\alpha}^{-1}(0) = 0$, we deduce that  $\partial \Sigma_{t,\alpha}$ divides $\mathcal{C}_{t,\alpha}(\lambda_0)$ into two Jordan domains and that the part of $\partial \Sigma_{t,\alpha}$ lying inside $\mathcal{C}_{t,\alpha}(\lambda_0)$ is mapped under $f_{t,\alpha}$ onto the part of $\partial \Omega_{t,\alpha}$ lying inside $\mathcal{C}(\lambda)$ and the same holds for the complementary parts lying outside $\mathcal{C}_{t,\alpha}(\lambda_0)$ and $\mathcal{C}(\lambda)$ respectively, then $\lambda_0 \notin \overline{\Sigma}_{t,\alpha}$.

Finally, let $|\lambda| \geq e^{(\alpha-1)t/2}$ and recall that $\lambda_0 \in 1/\overline{\Gamma}_{t,\alpha}$. Then we shall prove that for any $\alpha \in (0,1)$,
\begin{equation}\label{EmptyInter}
\partial \Sigma_{t,\alpha} \cap \left(1/\partial \Gamma_{t,\alpha}\right) = \emptyset. 
\end{equation}
To this end, recall from \cite{Biane1} that  
\begin{equation*}
1/\partial \Gamma_{t,\alpha} \subset \left\{z,  |f_{\alpha t,1}(1/z)|= 1\right\}
\end{equation*}
and that the image of this Jordan curve under the map $z \mapsto \chi = (1+z)/(1-z)$ lies in the left half-plane $\{\Re(\chi) \leq 0\}$ and circles $\chi = -1$. If 
\begin{equation*}
w = \frac{2\alpha - 1+z}{1-z} 
\end{equation*}
then 
\begin{equation*}
w = \alpha \chi + \alpha-1,
\end{equation*}
so that the image of $1/\partial \Gamma_{t,\alpha}$ under the M\"obius transformation $z \mapsto w$  lies in the left half-plane $\{\Re(w) \leq \alpha-1\}$. But if 
\begin{equation*}
w= x+iy \in 1/\partial \Gamma_{t,\alpha},
\end{equation*}
then 
\begin{equation} \label{Gamm}
\frac{(x+1-2\alpha)^2 + y^2}{(x+1)^2 + y^2} = e^{-t(x+2(1-\alpha))}.
\end{equation}
Since the right-hand side is a positive real less than one then $2(\alpha-1) < x \leq \alpha-1$. Now, the proof of Lemma \ref{Jordan} shows that the curve $\partial \Sigma_{t,\alpha}$ is parametrized by: 
\begin{equation}
y^2_{\Sigma} = \frac{(x+1)^2-(x^2+2x(1-2\alpha)+1)e^{tx}}{e^{tx}-1}, 
\end{equation}
while \eqref{Gamm} shows that the image of $1/\partial \Gamma_{t,\alpha}$ under the M\"obius transformation $z \mapsto w$ is parametrized by: 
\begin{equation*}
y^2_{1/\Gamma} = \frac{(x+1)^2-(x+1-2\alpha)^2e^{t(x+2(1-\alpha))}}{e^{t(x+2(1-\alpha))}-1}. 
\end{equation*}
Writing $x^2+2x(1-2\alpha)+1 =  (x+1-2\alpha)^2 + 4\alpha(1-\alpha)$, we readily derive: 
\begin{align*}
y^2_{\Sigma} - y^2_{1/\Gamma} & = \frac{-4\alpha x e^{tx} + 4\alpha(x+1-\alpha) e^{t(x+2(1-\alpha))}-4\alpha(1-\alpha)e^{tx}e^{t(x+2(1-\alpha))}}{(e^{tx}-1)(e^{t(x+2(1-\alpha))}-1)}
\\& = -4\alpha e^{tx}\frac{x-(x+1-\alpha)e^{2t(1-\alpha)} + (1-\alpha)e^{t(x+2(1-\alpha))}}{(e^{tx}-1)(e^{t(x+2(1-\alpha))}-1)}
\\&=\frac{-4\alpha e^{tx}\psi_{t,\alpha}(x)}{(e^{tx}-1)(e^{t(x+2(1-\alpha))}-1)}. 
\end{align*}
The denominator is negative for any $2(\alpha-1) < x \leq \alpha-1$. Besides, by Lemma \ref{numerator}, the numerator is negative on $[2(\alpha-1) ,\alpha-1]$ and in turn $y^2_{\Sigma} > y^2_{\Gamma}$. Consequently, \eqref{EmptyInter} holds so that if $|\lambda| \geq e^{(\alpha-1)t}$ then $\lambda_0 \notin \overline{\Sigma}_{t,\alpha}$.
\end{itemize}
\end{proof}  

We now work toward finding the expression of $s_{t,\alpha}(\lambda)$ for $\lambda$ outside $\overline{\Omega}_{t,\alpha}$. Since $s_{t,\alpha}$ is defined as the limit of $S(t,\lambda,x)$ as $x$ tends to zero with $\lambda$ fixed, we then wish to use the expression \eqref{S} of $S$ along curves $(t,\lambda(t),x(t))$. However, there are to difficulties with this argument: the first one is that the PDE for $S$ holds only when $x>0$, so we are not allowed to simply set $x_0=0$ in the formula \eqref{S}. The other difficulty is that when $x_0\rightarrow 0^+$, $\lambda(t)$ is not fixed since it depends on $x_0$.
In order to overcome these two difficulties, we will show that $S$ has a continuous extension to a neighborhood of $(t,\lambda,0)$ in the variables $\lambda$ and $x$ (it is for this reason that we have allowed $x_0$ to be slightly negative).
To that end, we consider the map
\begin{equation*}
V_{t,\alpha}(\lambda_0,x_0)=\left( \lambda_{t,\alpha}(\lambda_0,x_0),x_{t,\alpha}(\lambda_0,x_0) \right).
\end{equation*}
The domain of $V_{t,\alpha}$ consists of pairs $(\lambda_0,x_0)$ such that $\lambda_0\in\left(\overline{\Gamma}_{t,\alpha}\cup1/\Gamma_{t,\alpha}\right)\cap\overline{\Sigma}_{t,\alpha}^c$ and 
\begin{equation*}
p_x(0)=\tau(q_0)=\frac{\alpha}{x_0+|1-\lambda_0|^2}+\frac{1-\alpha}{x_0+|\lambda_0|^2}
\end{equation*}
 is positive. We note that under the last condition, $V_{t,\alpha}$ is well defined even when $x_0$ is slightly negative (see Remark \ref{x0negative}).

\begin{lem}\label{regularity}
If $\lambda_0\in\overline{\Gamma}_{t,\alpha}\cup1/\Gamma_{t,\alpha}$ is not in $\overline{\Sigma}_{t,\alpha}$, the Jacobin matrix of $V_{t,\alpha}$ at $(\lambda_0,0)$ is invertible. 
\end{lem}
\begin{proof}
Let $\lambda_0=r_0e^{i\theta_0}$. We note that when $\theta_0$ and $r_0$ vary while $x_0$ remains 0, then $x_{t,\alpha}$ remains equal to zero, so that 
\begin{equation*}
\partial_{\theta_0}x_{t,\alpha}(\lambda_0,0)=\partial_{r_0}x_{t,\alpha}(\lambda_0,0)=0,
\end{equation*}
We note also that when $x_0=0$, we have
\begin{equation*}
\lambda_{t,\alpha}(\lambda_0,x_0)=|\lambda_{t,\alpha}(\lambda_0,0)|e^{i\theta_{t,\alpha}(\lambda_0,0)}=f_{t,\alpha}(\lambda_0).
\end{equation*}
Hence, the Jacobin of $V_{t,\alpha}$ at $(\lambda_0,0)$ takes the following form
\begin{equation*}
\det V_{t,\alpha}'= \left|
            \begin{array}{cc}
                J &    *   \\
            0 &  \partial_{x_0}x_{t,\alpha} \\
            \end{array}
          \right|,
\end{equation*}
where $J$ is the Jacobian matrix of the map $f_{t,\alpha}$.
On the one hand, since $\lambda_0\in\overline{\Sigma}_{t,\alpha}^c$, we have
\begin{equation*}
T_\alpha(\lambda_0)=\frac{1}{p_x(0)}>\int_0^t|\lambda_{s,\alpha}(\lambda_0,0)|^2ds,
\end{equation*}
and hence
\begin{equation*}
\partial_{x_0}x_{t,\alpha}(\lambda_0,0)=\left(1-p_x(0)\int_0^t|\lambda_{s,\alpha}(\lambda_0,0)|^2ds\right)^2>0.
\end{equation*}
On the other hand from Remark \ref{homeomorphism}, the function $f_{t,\alpha}$ is injective on $\overline{\Gamma}_{t,\alpha}\cup1/\Gamma_{t,\alpha}$ so that $f_{t,\alpha}'(\lambda_0)$ is nonzero and hence $J$ is invertible.
Thus, the Jacobin of $V_{t,\alpha}$ at $(\lambda_0,0)$ is nonzero.
\end{proof}

Combining Theorem \ref{expression} and Theorem \ref{surj-outside} together with Lemma \ref{regularity}, we get: 
\begin{cor}\label{outside}
For any $t>0$  and any $\alpha\in (0,1)$, if $\lambda\ne0$ lies outside $\overline{\Omega}_{t,\alpha}$ and $|\lambda| \neq e^{(\alpha-1)t/2}$ then,
\begin{multline*}
s_t( \lambda)
=\alpha \log|1-\lambda_0|^2+(1-\alpha)\log|\lambda_0|^2
\\ +\frac{1}{2}\left(\frac{\alpha }{1-\lambda_0}+\frac{\alpha }{1-\overline{\lambda_0}}-\frac{ \alpha^2}{(1-\lambda_0)^2}-\frac{ \alpha^2}{(1-\overline{\lambda_0})^2}-1 \right)t
+  \log\left|\frac{\lambda}{\lambda_0}\right|,
\end{multline*}
where $\lambda_0=f_{t,\alpha }^{-1}\left(\lambda\right)$.
\end{cor}
\begin{proof}
Define $\text{HJ}$ by the right-hand side of the formula \eqref{S}:
\begin{equation*}
\text{HJ}(u,\lambda_0,x_0)=\tau(\log(|P-\lambda_0|^2+x_0))+\left(H_0-\frac{1}{2}\right)u+\log|\lambda_{u,\alpha}(\lambda_0,x_0)|-\log|\lambda_0|
\end{equation*}
so that,
\begin{equation}\label{S1}
S\left(u, \lambda_{u,\alpha}(\lambda_0,x_0),x_{u,\alpha}(\lambda_0,x_0) \right)=\text{HJ}(u,\lambda_0,x_0).
\end{equation}
Then for small and positive $x$, we have
\begin{equation*}
S\left(u, \lambda,x \right)=\text{HJ}\left(u,\big(V_{u,\alpha}^{-1}(\lambda,x)\big)\right).
\end{equation*}
where by definition
\begin{equation*}
 \lambda_{u,\alpha}\big(V_{u,\alpha}^{-1}(\lambda,x)\big)=\lambda \quad {\rm and}\ x_{u,\alpha}\big(V_{u,\alpha}^{-1}(\lambda,x)\big)=x.
\end{equation*}
Now, by the inverse function theorem, if $x$ tends to zero from above with $\lambda$ fixed, then $V_{u,\alpha}^{-1}(\lambda,x)$ tends to $(f_{u,\alpha}^{-1}(\lambda),0)=(\lambda_0,0)$. 
Thus, the function
\begin{equation*}
s_t( \lambda) :=\lim_{x\rightarrow0}S(t,\lambda,x)
\end{equation*}
can be computed by putting $\lambda_{u,\alpha}(\lambda_0,x_0)=\lambda$ and letting $x_0$ tend to zero in the expression  \eqref{S1}:
\begin{align*}
s_t( \lambda) & = \tau[\log\big(|P-\lambda_0|^2\big)] +\left(H_0-\frac{1}{2}\right)t +\log\left|\frac{\lambda}{\lambda_0}\right|.
\end{align*}
By Remark \ref{homeomorphism}, we have $\lambda_0 \in \Gamma_{t,\alpha} \cup 1/\Gamma_{t,\alpha}$, in particular  $\lambda_0 \neq 1$. Besides, we have $\lambda_0 \neq 0$ since $\lambda \neq 0$. Consequently, the spectral Theorem allows to write:  
\begin{equation*}
\tau[\log\big(|P-\lambda_0|^2\big)] = \alpha \log|1-\lambda_0|^2+(1-\alpha)\log|\lambda_0|^2.
\end{equation*}
On the other hand, if $x_0$ tend to zero, the Hamiltonian reduces to: 
\begin{align*}
H_0& =  K_1(1-K_1) + \frac{K_2^2}{4}
\\& = \frac{(a_0p_a(0) + b_0p_b(0))(2-a_0p_a(0)-b_0p_b(0))}{4} + \frac{(a_0p_b(0)-b_0p_a(0))^2}{4}
\\& = \frac{(1-2\dot{v}(0))(1+2\dot{v}(0)) + (a_0p_b(0)-b_0p_a(0))^2}{4}. 
\end{align*}
Using the formulas, 
\begin{equation*}
1-2\dot{v}(0) = 2- 2\alpha - 2\alpha \frac{a_0- |\lambda_0|^2}{|\lambda_0-1|^2},
\end{equation*}
\begin{equation*}
1+2\dot{v}(0) = 2\alpha + 2\alpha \frac{a_0- |\lambda_0|^2}{|\lambda_0-1|^2},
\end{equation*}
\begin{equation*}
\frac{(a_0p_b(0)-b_0p_a(0))^2}{4} = \frac{\alpha^2b_0^2}{|\lambda_0-1|^4}
\end{equation*}
we end up with: 
\begin{align*}
H_0 - \frac{1}{2} = \frac{1}{2}\left(\frac{\alpha }{1-\lambda_0}+\frac{\alpha }{1-\overline{\lambda_0}}-\frac{ \alpha^2}{(1-\lambda_0)^2}-\frac{ \alpha^2}{(1-\overline{\lambda_0})^2} -1 \right).
\end{align*}
Together with Proposition \ref{surj-outside} yield: for any $\lambda \in \mathbb{C} \setminus \overline{\Omega}_{t,\alpha},\lambda\ne0$,
\begin{align*}
s_t( \lambda) &  = \alpha \log|1-\lambda_0|^2+(1-\alpha)\log|\lambda_0|^2 +\frac{1}{2}\left(\frac{\alpha }{1-\lambda_0}+\frac{\alpha }{1-\overline{\lambda_0}}-\frac{ \alpha^2}{(1-\lambda_0)^2}-\frac{ \alpha^2}{(1-\overline{\lambda_0})^2}-1 \right)t
\\& +\log\left|\frac{\lambda}{\lambda_0}\right|.
\end{align*}
\end{proof}
We are now ready to prove our main result which asserts that $\Delta s_{t}(\lambda)=0$ in distributional sense for $\lambda \in \mathbb{C} \setminus \{\overline{\Omega}_{t,\alpha}\cup\{0\}\}$.

\begin{proof}[Proof of Theorem \ref{support}]
From Proposition \ref{outside}, the function
\begin{equation*}
\lambda \mapsto s_t( \lambda) - (1-\alpha)\log|\lambda_0|^2 - \alpha\log|1-\lambda_0|^2 
\end{equation*}
 is the real part of a holomorphic function in $\mathbb{C} \setminus \{\overline{\Omega}_{t,\alpha}\cup\{0\}\}$, and is therefore harmonic there. Morever, since $\lambda_0 \notin \{0,1\}$ then the linear combination
 \begin{equation*}
\lambda \mapsto (1-\alpha)\log|\lambda_0|^2 + \alpha\log|1-\lambda_0|^2
\end{equation*}
is also harmonic in the same domain. Since the circle of radius $e^{(\alpha-1)t/2}$ has (two-dimensional) zero Lebesgue measure, then the Theorem is proved. 
 \end{proof}

\begin{rem}
When $\alpha=1$, the region  $\Omega_{t,1}$ becomes the closed unit disc since the boundary of $\Sigma_{t,1}$ maps to the unit circle under $f_{t,1}$. In particular, $\Omega_{t,1}$ contains the support of $Y_t$. On the other hand, if $|\lambda| > 1 > e^{(\alpha-1)t}$ then $\lambda_0 \in 1/\Gamma_{t,\alpha}$ and $\Delta s_{t}(\lambda)=0$ in the strong sense. This is in agreement with the general theory since the spectrum of $Y_tP$ is contained in the closed unit disc.  
\end{rem}

\begin{figure}
\begin{center}
\includegraphics[width=7cm]{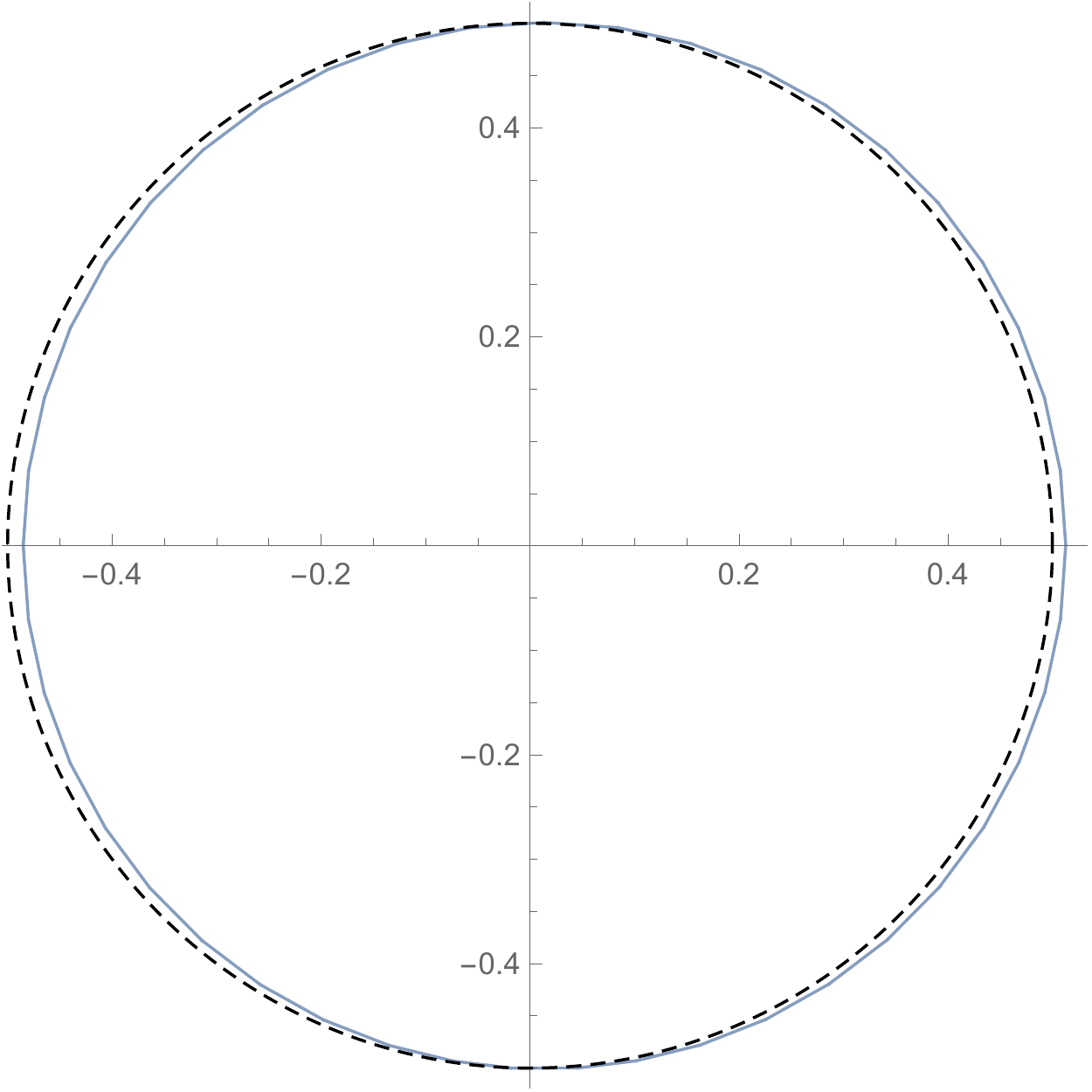} 
\end{center}
\caption{The region $\Omega_{t,\alpha}$ for $(t,\alpha)=(8,0.25)$ with the circle $C(0,0.5)$ (dashed) indicated for comparison}.
\end{figure}

We close the paper by the following result showing that $\partial\Omega_{t,\alpha}$ approaches the circle $\mathbb{T}(0,\sqrt{\alpha})$ when $t$ approaches infinity. This in agreement with Haagerup and Larsen result on the Brown measure of $PUP$.
\begin{pro}
For fixed $t>0$ and $\alpha\in(0,1)$, let $z_{t,\alpha}$ denote a boundary point of $\Sigma_{t,\alpha}$. Then, 
\begin{equation*}
\lim_{t\rightarrow\infty}|f_{t,\alpha}(z_{t,\alpha})|=\sqrt{\alpha}.
\end{equation*}
\end{pro}
\begin{proof}
By definition, $z_{t,\alpha}\ne1$ and satisfies
\begin{equation*}
t=T_{\alpha}(z_{t,\alpha})=\frac{|1-z_{t,\alpha}|^2}{\alpha(|z_{t,\alpha}|^2-1)+(1-\alpha)|z_{t,\alpha}|^2}\log\left(\frac{\alpha|z_{t,\alpha}|^2+(1-\alpha)|1-z_{t,\alpha}|^2}{\alpha} \right).
\end{equation*}
But, since for all $x>0$
\begin{equation*}
\log(x)\le x-1, 
\end{equation*}
it follows that,
\begin{equation*}
\alpha t\le |1-z_{t,\alpha}|^2\le 1+|z_{t,\alpha}|^2.
\end{equation*}
Hence, for $t$ approaching infinity we have
\begin{equation*}
|1-z_{t,\alpha}|^2 \sim  |z_{t,\alpha}|^2 \rightarrow\infty,
\end{equation*}
and therefore, we obtain
\begin{equation*}
|f_{t,\alpha}(z_{t,\alpha})|^2=\frac{\alpha|z_{t,\alpha}|^2}{\alpha|z_{t,\alpha}|^2+(1-\alpha)|1-z_{t,\alpha}|^2} \sim \alpha .
\end{equation*}
\end{proof}

\end{document}